\documentclass[3p,preprint,number]{elsarticle} 

\usepackage{rotating}
\usepackage{amsmath,amsthm,amssymb,amsfonts,amscd, array}
\usepackage{natbib}

\usepackage[boxruled,linesnumbered]{algorithm2e}
\usepackage[caption=false]{subfig}
\usepackage{nomencl}
\usepackage{etoolbox}
\usepackage{wrapfig}
\usepackage{longtable}
\usepackage{multirow}

\newtheorem{thm}{Theorem}

\usepackage{enumitem}
\usepackage[commandnameprefix=always]{changes}

\theoremstyle{plain}

\newtheorem{prop}[thm]{Proposition}

\usepackage{accents}

\newcounter{Propcount}
\newproof{pf}{Proof}

\usepackage{multirow}

\usepackage[noend]{algpseudocode}
\makeatletter
\def\BState{\State\hskip-\ALG@thistlm}
\makeatother

\biboptions{authoryear}

\begin{document}

\begin{frontmatter}

\title{Approximate dynamic programming for profit estimation of connected hydro reservoirs}

\author[label1]{Farzaneh Pourahmadi\renewcommand*{\thefootnote}{$*$}\footnote{Corresponding author.\\ \emph{E-mail addresses:} fp@math.ku.dk}}
\author[label1]{Trine Krogh Boomsma\renewcommand*{\thefootnote}{}}

\address[label1]{Department of Mathematical Sciences, University of Copenhagen, Universitetsparken 5, 2100 K\o benhavn \O, Denmark}

\begin{abstract}
In this paper, we study the operational problem of connected hydro power reservoirs which involves sequential decision-making in an uncertain and dynamic environment. The problem is traditionally formulated as a stochastic dynamic program accounting for the uncertainty of electricity prices and reservoir inflows. This formulation suffers from the curse of dimensionality, as the state space explodes with the number of reservoirs and the history of prices and inflows. To avoid computing the expectation of future value functions, the proposed model takes advantage of the so-called post-decision state. To further tackle the dimensionality issue, we propose an approximate dynamic programming approach that estimates the future value of water using a linear approximation architecture. When the time series of prices and inflows follow autoregressive processes, our approximation provides an upper bound on the future value function. We use an offline training algorithm based on the historical data of prices and inflows and run both in-sample and out-of-sample simulations. Two realistic test systems of cascade and network connected reservoirs serve to demonstrate the computational tractability of our approach. In particular, we provide numerical evidence of convergence and quality of solutions. For our test systems, our results show that profit estimation is improved by 20\% when including inflows in the linear approximation. 
\end{abstract}


\end{frontmatter}
  
\section{Introduction}

With the rising penetration of renewable resources in many power systems, hydro-power plants are playing an increasingly important role as large-scale flexible units. The operation of hydro-power systems is a complex stochastic and dynamic optimization problem, involving sequential decisions under uncertainty. The coordination of water releases from multiple connected reservoirs over time may pose serious challenges. The modeling of uncertainty in electricity prices and external water inflows further increases complexity. 

The existing literature typically describes the operational problem of a hydro-power plant as a multi-stage stochastic program, using a scenario tree to characterize uncertainty \cite{Trine,Trine2,stochastic}. A drawback of this approach is the large number of scenarios required to accurately represent the distribution of uncertainty. The size of the scenario tree, however, increases exponentially with the number of stages, which may result in computational intractability. An alternative to decision-making under uncertainty is multi-stage robust optimization. This method determines an optimal solution with respect to the worst-case realization of an uncertainty set \cite{robust1,robust2}, possibly producing an overly conservative solution. Furthermore, it is non-trivial to construct an uncertainty set that includes all potential distributions in a multi-dimensional space.

A different technique for sequential decision-making is based on dynamic programming, using the Bellman equations \cite{Bellman}. With the introduction of state variables, the principle of Bellman allows the hydro-power problem to be solved recursively. Often, the time horizon consists of a finite number of time periods, referred to as stages, and the state space is discretized into finite number of values for each variable and in each time period. The value of being in a state includes the immediate return of the current state and the expected future value, also called profit-to-go \cite{Dimitri}. Applying this methodology, the optimization problem decomposes into stage-wise sub problems and computational complexity scales linearly with the number of stages \cite{Bellman,Dimitri}. Nevertheless, the Bellman equations for hydro-power operation may not be solved to optimality, as the state space easily explodes with the number of reservoirs and the history of prices and inflows.

Approximate dynamic programming (ADP) offers various strategies to overcoming the curse of dimensionality such as simulation of the state space and approximation of the future value function \cite{Powelbook,labadi}. For the majority of the literature in this domain, approximation relies on the discrete representation of the state space \cite{Reliability,discrete,discrete2}. This type of approximation may be inaccurate or intractable for large-scale problems with a substantial number of states. Other ADP algorithms are based on linear and non-linear approximations of the value function. One of the most widely used approaches, the neural network framework, deploys a complex nonlinear function, which generally does not provide any optimality guarantee and interpretability \cite{DimitriNeuro}. Other types of non-linear approximations \cite{Powel1,Powel2} suffer from similar lack of guarantees. In contrast, linear approximations may produce linear programming sub problems that can be solved to optimality. A way to obtain a piece-wise linear approximation in value space is by stochastic dual dynamic programming \cite{Nils,Benjamin,Philip,Flach}. This method approximates a convex future value function by a collection of supporting hyperplanes, representing an outer approximation. However, an accurately estimate of the value function may require many hyperplanes. Also, obtaining these hyperplanes requires both a forward and backward pass in the algorithm. Compared to piece-wise linear approximations, e.g.\ obtained by duality, the use of a linear approximation architecture is less computationally expensive and learning may be obtained only by a forward pass of an algorithm.  

To overcome the aforementioned challenges, we propose a novel and tractable ADP framework for operation of connected hydro reservoirs. We address what and how to learn from historical data to accurately estimate future profit and make sequential decisions under uncertainty. Decisions relate to the amounts of water released from multiple connected reservoirs, and states include the reservoir level, current and past electricity prices and inflows. Our model exploits a powerful strategy based on the so-called the post-decision state to avoid the computation of the expectation in the Bellman equations. The post-decision state captures the state of the system immediately after making a decision but before any new exogenous information arrives. To further tackle the curse of dimensionality, we replace the future value of water by a linear approximation learnt from samples of random prices and inflows. Firstly, the linearity of the value function allows for the stage-wise sub problems to be solved as linear programs with an optimality guarantee. Secondly, the linear approach easily generalizes to more advanced modeling of reservoir operation by including additional linear constraints on feasible decisions. We show that in case the time series of prices and inflows follow an autoregressive process, the approximation provides an upper bound on future profits. On this basis, we propose an offline learning process to train an online model. The framework may be used to assist the decision-making of reservoir owners participating in the wholesale market. 

We assess the performance of our model using both in-sample and out-of-sample simulations. We provide numerical evidence of convergence and quality of solutions for two realistic case studies. In particular, we establish convergence of the value function towards its true value for the deterministic problem. For the stochastic problem, the function converges in the sense that variations in its value decrease with the number of samples. Considering the optimal solutions, we confirm that when the price is low, water is stored such that when the price is higher, the hydro plants generate electricity. To assess solution quality, we compare in-sample and out-of-sample values, finding a difference of less than 2\%. To further evaluation, we compare to the so-called wait-and-see solutions, revealing an estimated value of perfect information of less than 10\%. Most importantly, our results demonstrate that accurate estimation of the future profit depends on not only the current reservoir level but also on the estimation of future inflows. For our test systems, our test cases show that profit estimation is improved by 20\% when including inflows in the linear approximation.

The remainder of this paper is organized as follows. Section \ref{2} provides the model for hydro reservoir operation and formulates it as a stochastic dynamic program which is reformulated and approximated in Sections \ref{3} and \ref{4}, respectively. Section \ref{5} describes an offline algorithm for training the approximated model. Sections \ref{6} and \ref{7} present numerical results for two realistic case studies. Finally, Section \ref{8} concludes the paper.  

\section{Modeling hydro reservoir operation} \label{2}

A hydro power plant consists of multiple interconnected reservoirs. Operational flexibility implies that water can be released from elevated reservoirs and led through a power station with a number of turbines, converting its potential energy into power, at times of high demand for electricity. Likewise, the reservoirs can store natural water inflows or energy can be used to pump back water into the reservoirs at times of no or low demand for electricity. For reservoirs in a cascade, water releases from upstream reservoirs usually contribute to downstream inflows and pumping from downstream power stations results in upstream inflow. The owner of a hydro power plant use of this flexibility to maximize profit. We consider a price-taking producer facing the development in hourly electricity market prices and adapting generation accordingly over a finite time horizon of a number of days.

The operation of the reservoirs entails a large number of sequential decisions as well as considerable uncertainty. The problem involves reservoir storage dynamics, which should be incorporated into the water policy. Moreover, as charging and discharging of each reservoir influence the reservoir level of the others, the decisions of water release and pumping from multiple interconnected reservoirs requires a coordinated policy. We consider electricity market prices and reservoir inflows as the main sources of uncertainty due to unexpected market circumstances and unforeseen weather conditions that are disclosed over time. We model the operational problem of the reservoirs by stochastic dynamic programming such that the value of current decisions in each stage is weighted against their future effects. In our model, each hour represents a stage, decisions relate to the amounts of water charging and discharging, and the states include the reservoir levels, electricity prices, and random inflows. In the following sections, we reformulate, approximate and solve this problem using approximate dynamic programming. 

We start by defining relevant notation. The time horizon $\{1,...,T\}$ is taken to be a few operation days discretized into hourly time intervals indexed by $t$. We consider a hydro power network of interconnected reservoirs and index a reservoir by $j$ and the set of reservoirs by $J$. We let $\textbf{l}_{t}=(l_{1t},\dots,l_{|J|t})^T\in\mathbb{R}^{|J|}$ be the storage levels of the reservoirs in the beginning of time period $t$, where $|.|$ is the cardinality operator. For now, we disregard pumping of water to the reservoirs such that decisions only relates to water discharging. Accordingly, we let the decision vectors $\boldsymbol{\pi}_{t}=(\pi_{1t},\dots,\pi_{|J|t})^T\in\mathbb{R}^{|J|}$ represent the discharges from reservoirs during time period $t$. The random vectors $\boldsymbol{\nu}_{t}=(\nu_{1t},\dots,\nu_{|J|t})^T\in\mathbb{R}^{|J|}$ and variables $\rho_t\in\mathbb{R}$ refer to natural inflows of the reservoirs and the electricity market price during time period $t$, respectively. Also, we let $\boldsymbol{\nu}_{j[t]}=(\nu_{j1},\dots,\nu_{jt})^T\in\mathbb{R}^{t}$, $\boldsymbol{\nu}_{[t]}=(\boldsymbol{\nu}_{1[t]},\dots,\boldsymbol{\nu}_{|J|[t]})\in\mathbb{R}^{t|J|}$ and $\boldsymbol{\rho}_{[t]}=(\rho_{1},\dots,\rho_t)^T\in\mathbb{R}^{t}$ hold the time series of inflows and prices up to time $t$. We assume that the realizations of $\boldsymbol{\nu}_{[t]}$ and $\rho_{[t]}$ are known at the time of making decisions $\boldsymbol{\pi}_t$. To model the capacities of the reservoirs, we introduce upper and lower bounds on the reservoir levels, denoted by $\textbf{l}^{max}\in\mathbb{R}^{|J|}$ and $\textbf{l}^{min}\in\mathbb{R}^{|J|}$, and likewise upper and lower bounds on their discharge levels, represented by $\boldsymbol{\pi}^{max}\in\mathbb{R}^{|J|}$ and $\boldsymbol{\pi}^{min}\in\mathbb{R}^{|J|}$, respectively.

For ease of exposition, we first consider a cascade of connected reservoirs. Later, we generalize the problem to a more complex network of reservoirs. At time $t$, the set of feasible water discharges is given by 
\begin{align*}
\Pi_t(\textbf{l}_{t},\boldsymbol{\nu}_{t})=\Big\{\boldsymbol{\pi}_{t}:\ &\textbf{l}_{t+1}=\textbf{l}_{t}+R\boldsymbol{\pi}_{t}+\boldsymbol{\nu}_{t},\textbf{l}^{min}\leq \textbf{l}_{t+1}\leq \textbf{l}^{max}, \boldsymbol{\pi}^{min}\leq \boldsymbol{\pi}_{t}\leq \boldsymbol{\pi}^{max}\Big\}, \ t=1,\dots,T,
\end{align*}
where $R\in \mathbb{R}^{|J|}\times\mathbb{R}^{|J|}$ represents connections between reservoirs such that $R_{jj}=-1, R_{jk}=1$ for $k\in J^-(j)$ and $R_{jk}=0$, otherwise. The set $J^-(j)$ denotes the reservoirs immediately upstream from reservoir $j$ with $J^-(0)=\emptyset$. We assume that downstream inflows from upstream reservoirs arrive at the same time as being discharged. If there is delay of upstream discharges, the state space must be extended. The first constraint enforces the reservoir balance and determines the next state of the reservoir level as a function of the current. The upper and lower limits for the reservoir level and water discharges are imposed in the second and third constraints, respectively. 

The function $G(\boldsymbol\pi_{t})$ determines the power generation level as a function of the water discharges. For simplicity, we assume that $G(\boldsymbol\pi_{t})$ is a linear function of $\boldsymbol\pi_{t}$ (the following analysis in fact applies for convex functions) such that $G(\boldsymbol\pi_{t})=\textbf{g}^T\boldsymbol\pi_{t}$ where $\mathbf{g}=(g_1,\dots,g_{|J|})^T$ determines the conversion rates from water to power. This assumption is valid if each reservoir is connected to a single power station or to multiple power stations with the same conversion rates. At time $t$, the profit function is denoted $C_t(\boldsymbol{\pi}_{t},\rho_{t})$ and is given by
\begin{align*}
C_t(\boldsymbol{\pi}_{t},\textbf{l}_{t},\boldsymbol{\nu}_{[t]},\boldsymbol{\rho}_{[t]})=\rho_{t}G(\boldsymbol{\pi}_{t}), \ t=1,\dots,T-1
\end{align*}
and 
\begin{align*}
C_T(\boldsymbol{\pi}_{T},\textbf{l}_{T},\boldsymbol{\nu}_{[T]},\boldsymbol{\rho}_{[T]})=\rho_{T}G(\boldsymbol{\pi}_{T})+\mathbb{E}\Big[\rho_{T+1} \Big|\boldsymbol{\rho}_{[T]}\Big]G({\textbf{l}}_{T+1}),
\end{align*}
where ${\textbf{l}}_{T+1}=\textbf{l}_{T}+R\boldsymbol{\pi}_{T}+\boldsymbol{\nu}_{T}$. If $\rho_{T+1}$ is a random future value that reflects power prices beyond the time horizon, the profit at time $T$ includes the expected future value of water in the reservoirs.

The problem is to determine feasible levels of water discharges $(\boldsymbol{\pi}_1,...,\boldsymbol{\pi}_T)$ that maximize expected accumulated profits over the time horizon, i.e. 
\begin{align*}
\max_{(\boldsymbol{\pi}_1,\dots,\boldsymbol{\pi}_T)\in \Pi_1\times\dots\times\Pi_T}\mathbb{E}\Big[\sum_{t=1}^T\ C_t(\boldsymbol{\pi}_{t},\textbf{l}_{t},\boldsymbol{\nu}_{[t]},\boldsymbol{\rho}_{[t]})\Big].
\end{align*}
where the expectation operator $\mathbb{E}[\cdot]$ is with respect to the joint distribution of $\boldsymbol{\nu}_{[T]}$ and $\boldsymbol{\rho}_{[T]}$. We require that the decisions $(\boldsymbol{\pi}_1,...,\boldsymbol{\pi}_T)$ are adapted to the stochastic process $\boldsymbol{\nu}_{1},\boldsymbol{\rho}_{1},\dots,\boldsymbol{\nu}_{T},\boldsymbol{\rho}_{T}$, i.e. that $\boldsymbol{\pi}_t$ depends on the realization of $\boldsymbol{\nu}_{[t]},\boldsymbol{\rho}_{[t]}$ but not on future realizations.

\subsection{Formulation by stochastic dynamic programming}

To formulate the operational problem of the hydro power network by stochastic dynamic programming, we let $(\textbf{l}_{t},\boldsymbol{\nu}_{[t]},\boldsymbol{\rho}_{[t]})$ be the so-called pre-decision state at time $t$, including the reservoir levels \emph{before} discharge decisions $\boldsymbol{\pi}_t$, also referred to as actions, are made. Moreover, we let $V_t(\textbf{l}_{t},\boldsymbol{\nu}_{[t]},\boldsymbol{\rho}_{[t]})$ denote the value of being in this state at time $t$. 

By the principle of optimality, the value functions satisfy the Bellman equations
\begin{align}
&V_{t}(\textbf{l}_{t},\boldsymbol{\nu}_{[t]},\boldsymbol{\rho}_{[t]})= \max_{\boldsymbol{\pi}_t\in \Pi_t(\textbf{l}_{t},\boldsymbol{\nu}_{t})}\Big\{C_t(\boldsymbol{\pi}_{t},\textbf{l}_{t},\boldsymbol{\nu}_{[t]},\boldsymbol{\rho}_{[t]})+\mathbb{E}\Big[V_{t+1}(\textbf{l}_{t+1},\boldsymbol{\nu}_{[t+1]},\boldsymbol{\rho}_{[t+1]})\Big|\boldsymbol{\nu}_{[t]},\boldsymbol{\rho}_{[t]}\Big]\Big\}, \nonumber\\ 
&t=1,\dots,T,\label{Bellman1}\\[2mm] 
&V_{T+1}(\textbf{l}_{T+1},\boldsymbol{\nu}_{[T+1]},\boldsymbol{\rho}_{[T+1]})=0,\label{Bellman2} 
\end{align}
where  $\mathbb{E}[\cdot|\cdot]$ is the conditional expectation. By these equations, the value at time $t$ depends on the current profit $C_t(\boldsymbol{\pi}_{t},\textbf{l}_{t},\boldsymbol{\nu}_{[t]},\boldsymbol{\rho}_{[t]})$ resulting from current actions $\boldsymbol{\pi}_{t}$ and the future value $V_{t+1}(\textbf{l}_{t+1},\boldsymbol{\nu}_{[t+1]},\boldsymbol{\rho}_{[t+1]})$, also referred to as the future water value, which is random at time $t$. 

Since the optimization problem of each stage involves an expected value in the objective function, it is a stochastic problem. Thus, the evaluation of an action involves the evaluation of the expectation. For instance, if the distribution is discrete with $N$ realizations, this requires the evaluation of $N$ future value functions. To avoid this, we use the post-decision state and reformulate the Bellman equations such that the optimal value of the optimization problem of each stage is random and the expectation is with respect to this optimal value. For each realization, it suffices to solve a deterministic optimization problem. This strategy is presented in the following section.

\section{Post-decision reformulation} \label{3}

We start by presenting the reformulation of \eqref{Bellman1}-\eqref{Bellman2} using the post-decision state. Let $\bar V_t$ be the value of the post-decision state $(\bar{\textbf{l}}_{t},\boldsymbol{\nu}_{[t]},\boldsymbol{\rho}_{[t]})$ at time $t$, including the reservoir level immediately \emph{after} making discharge decisions $\boldsymbol{\pi}_{t}$ but before the arrival of inflows $\boldsymbol{\nu}_{t}$, i.e.\ with $\bar{\textbf{l}}_{t}=\textbf{l}_{t}+R\boldsymbol{\pi}_{t}$, and thus, $\textbf{l}_{t+1}=\bar{\textbf{l}}_{t}+\boldsymbol{\nu}_{t}$. Then,
\begin{subequations} \label{reformulation}
\begin{align}
&\bar V_t(\bar{\textbf{l}}_{t},\boldsymbol{\nu}_{[t]},\boldsymbol{\rho}_{[t]})=\mathbb{E}\Big[V_{t+1}(\textbf{l}_{t+1},\boldsymbol{\nu}_{[t+1]},\boldsymbol{\rho}_{[t+1]}) \Big|\boldsymbol{\nu}_{[t]},\boldsymbol{\rho}_{[t]}\Big]\nonumber\\
&=\mathbb{E}\Big[\!\max_{\boldsymbol{\pi}_{t+1}\in \Pi_{t+1}(\textbf{l}_{t+1},\boldsymbol{\nu}_{t+1})}\!\Big\{C_{t+1}(\boldsymbol{\pi}_{t+1},\textbf{l}_{t+1},\boldsymbol{\nu}_{[t+1]},\boldsymbol{\rho}_{[t+1]})\!+\! \mathbb{E}\Big[V_{t+2}(\textbf{l}_{t+2},\boldsymbol{\nu}_{[t+2]},\boldsymbol{\rho}_{[t+2]})\Big|\boldsymbol{\nu}_{[t+1]},\boldsymbol{\rho}_{[t+1]}\Big]\!\Big\}\Big|\boldsymbol{\nu}_{[t]},\boldsymbol{\rho}_{[t]}\!\Big]\nonumber\\
&=\mathbb{E}\Big[\max_{\boldsymbol{\pi}_{t+1}\in \Pi_{t+1}(\bar{\textbf{l}}_{t}+\boldsymbol{\nu}_{t},\boldsymbol{\nu}_{t+1})}\Big\{C_{t+1}(\boldsymbol{\pi}_{t+1},\bar{\textbf{l}}_{t}+\boldsymbol{\nu}_{t},\boldsymbol{\nu}_{[t+1]},\boldsymbol{\rho}_{[t+1]})\!+\! \bar V_{t+1}(\bar{\textbf{l}}_{t+1},\boldsymbol{\nu}_{[t+1]},\boldsymbol{\rho}_{[t+1]})\Big\}\Big|\boldsymbol{\nu}_{[t]},\boldsymbol{\rho}_{[t]}\Big], \nonumber\\[2mm]
&t=0,\!\dots\!,\!T-1,\label{Bellman1post}\\[2mm]
&\bar V_{T}(\bar{\textbf{l}}_{T},\boldsymbol{\nu}_{[T]},\boldsymbol{\rho}_{[T]})=\mathbb{E}\Big[ V_{T+1}({\textbf{l}}_{T+1},\boldsymbol{\nu}_{[T+1]},\boldsymbol{\rho}_{[T+1]})\Big|\boldsymbol{\nu}_{[T]},\boldsymbol{\rho}_{[T]}\Big]=0.\label{Bellman2post}
\end{align}
\end{subequations}
The difference between pre-decision and post-decision states is illustrated for a decision-tree in Fig.\ \ref{fig1}. Solid lines correspond to discharge decisions and dotted lines to realizations of inflows. Square and circle nodes represent post-decision and pre-decision states, respectively. As seen from the figure, the pre-decision reservoir level $\mathbf{l}_t$ defines the state at time $t$ before we make decision $\boldsymbol{\pi}_t$. Then, the post-decision reservoir level $\bar{\mathbf{l}}_t$ defines the state at time $t$, immediately after we made the decision. Finally, the realization of the random vector $\boldsymbol{\nu}_t$ takes us to new pre-decision state $\mathbf{l}_{t+1}$.  

As in the pre-decision formulation, each stage involves an expected value. In the post-decision formulation, the expectation is with respect to an optimal value. if the distribution is discrete with $N$ realizations, each stage requires $N$ optimal values of deterministic optimization problems. Hence, we solve $N$ optimization problems in each stage. The evaluation of an action, however, involves only a single evaluation of the future value function. 

Since the strategy of using post-decision states mitigates the curse of dimensionality caused by evaluating an expectation, formulation \eqref{reformulation} may provide a computational advantage. In the remainder of the paper, we use the post-decision formulation.

\begin{figure}[tb]
\begin{center}
\includegraphics[height=2.5in,width=0.35\linewidth]{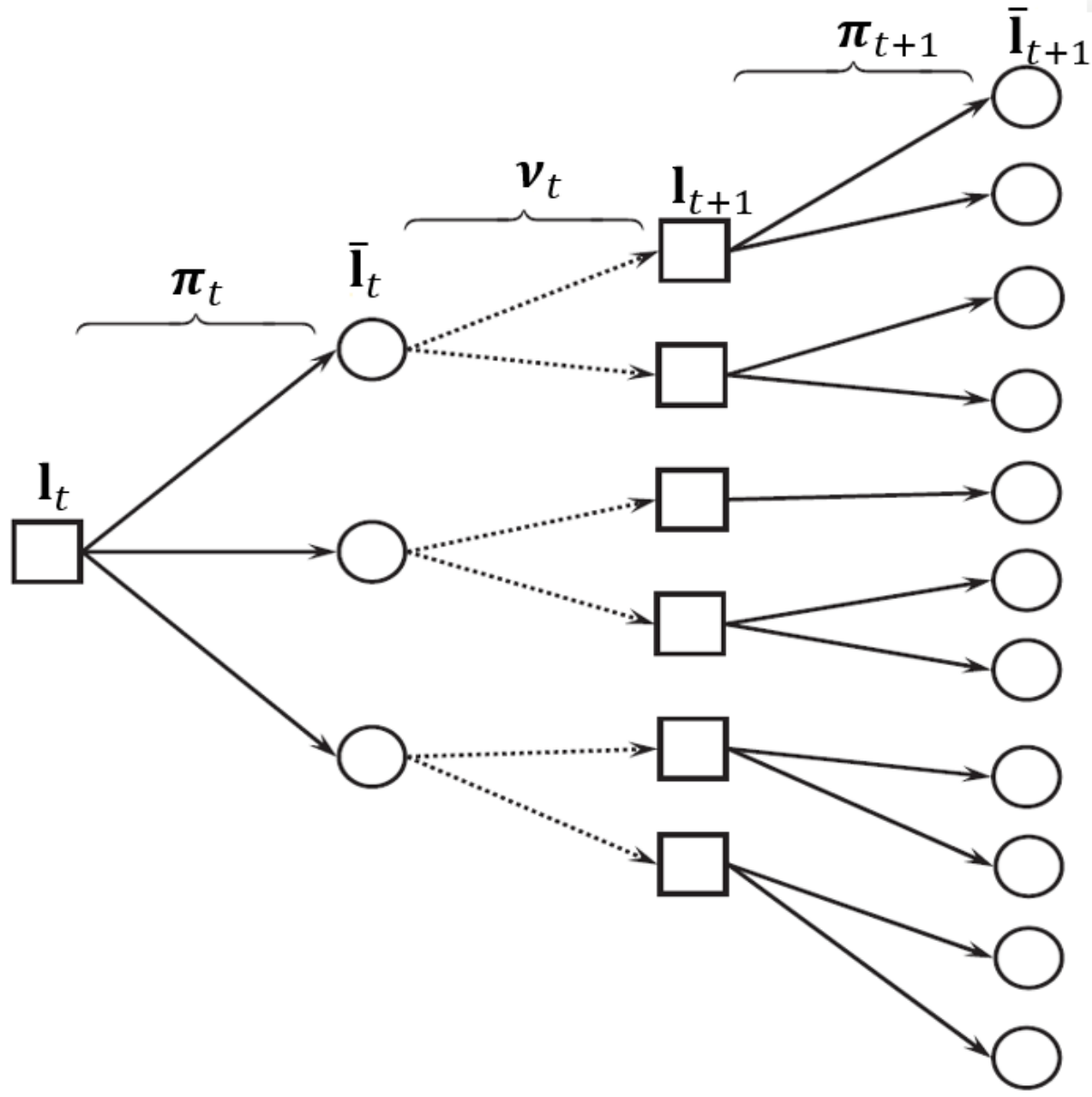}
\vspace{2mm}
 \caption{Post-decision and pre-decision states illustrated for a decision tree. Solid lines correspond to discharge decisions and dotted lines to realizations of inflows. Square and circle nodes represent post-decision and pre-decision states, respectively.} \label{fig1}
\end{center}
\end{figure}

\section{Value function approximation}\label{4}

For $t=0,\dots,T-1$, let
\begin{align}
&\bar V_t(\bar{\textbf{l}}_{t},\boldsymbol{\nu}_{[t]},\boldsymbol{\rho}_{[t]})=\mathbb{E}\Big[\bar W_t(\bar{\textbf{l}}_{t},\boldsymbol{\nu}_{[t+1]},\boldsymbol{\rho}_{[t+1]})\Big|\boldsymbol{\nu}_{[t]},\boldsymbol{\rho}_{[t]}\Big]\nonumber
\end{align}
with
\begin{align}
&\bar W_t(\bar{\textbf{l}}_{t},\boldsymbol{\nu}_{[t+1]},\boldsymbol{\rho}_{[t+1]})=\max_{\boldsymbol{\pi}_{t+1}\in \Pi_{t+1}(\bar{\textbf{l}}_{t}+\boldsymbol{\nu}_{t},\boldsymbol{\nu}_{t+1})}\Big\{C_{t+1}(\boldsymbol{\pi}_{t+1},\bar{\textbf{l}}_{t}+\boldsymbol{\nu}_{t},\boldsymbol{\nu}_{[t+1]},\boldsymbol{\rho}_{[t+1]})\!+\! \bar V_{t+1}(\bar{\textbf{l}}_{t+1},\boldsymbol{\nu}_{[t+1]},\boldsymbol{\rho}_{[t+1]})\Big\}, \nonumber
\end{align}
and
\begin{align}
\Pi_{t+1}(\bar{\textbf{l}}_{t}+\boldsymbol{\nu}_{t},\boldsymbol{\nu}_{t+1})=\Big\{\boldsymbol{\pi}_{t+1}:\ &\bar{\textbf{l}}_{t+1}=\bar{\textbf{l}}_{t}+\boldsymbol{\nu}_{t}+R\boldsymbol{\pi}_{t+1},\textbf{l}^{min}\leq \bar{\textbf{l}}_{t+1}+\boldsymbol{\nu}_{t+1}\leq \textbf{l}^{max}, \boldsymbol{\pi}^{min}\leq \boldsymbol{\pi}_{t+1}\leq \boldsymbol{\pi}^{max}\Big\}.\nonumber
\end{align}

In the following, we assume that the dynamics of prices and inflows are given by the moving average autoregressive (ARMA) processes
\begin{align}
{\rho}_{t+1}=\boldsymbol\theta_{[t]}^T{\boldsymbol{\rho}}_{[t]}+\boldsymbol{\eta}_{[t+1]}^T\boldsymbol{\epsilon}_{[t+1]}, \ {\nu}_{jt+1}=\boldsymbol{\psi}_{j[t]}^T\boldsymbol{\nu}_{j[t]}+\boldsymbol{\phi}_{j[t+1]}^T\boldsymbol{\xi}_{j[t+1]}\nonumber
\end{align}
with $\boldsymbol\theta_{[t]}=(\theta_{1},\dots,\theta_{t})^T, \boldsymbol\eta_{[t]}=(\eta_{1},\dots,\eta_{t})^T\in\mathbb{R}^t$ and $\epsilon_t\in\mathbb{R}$ i.d.d.\ random variables and with $\boldsymbol{\psi}_{j[t]}=(\psi_{j1},\dots,\psi_{jt})^T,\boldsymbol{\phi}_{j[t]}=(\phi_{j1},\dots,\phi_{jt})^T\in\mathbb{R}^{t}$ and ${\xi}_{jt}\in\mathbb{R}$ i.d.d for $j=1,\dots,|J|$. A compact form of the inflow time series is $\boldsymbol{\nu}_{t+1}=\text{diag}(\boldsymbol\psi_{[t]}^T\boldsymbol{\nu}_{[t]}+\boldsymbol\eta_{[t+1]}^T\boldsymbol{\xi}_{[t+1]})$ where $\boldsymbol{\psi}_{[t]}=(\boldsymbol{\psi}_{1[t]},\dots,\boldsymbol{\psi}_{|J|[t]}),\boldsymbol{\phi}_{[t]}=(\boldsymbol{\phi}_{1[t]},\dots,\boldsymbol{\phi}_{|J|[t]}),\boldsymbol{\xi}_{[t]}=(\boldsymbol{\xi}_{1[t]},\dots,\boldsymbol{\xi}_{|J|[t]})\in\mathbb{R}^{t|J|}$.

With this assumption 
\begin{align}
&\bar V_t(\bar{\textbf{l}}_{t},\boldsymbol{\nu}_{[t]},\boldsymbol{\rho}_{[t]})=\mathbb{E}\Big[\bar W_t(\bar{\textbf{l}}_{t},\boldsymbol{\nu}_{[t]},\text{diag}(\boldsymbol\psi_{[t]}^T\boldsymbol{\nu}_{[t]}+\boldsymbol\eta_{[t+1]}^T\boldsymbol{\xi}_{[t+1]}),\boldsymbol{\rho}_{[t]},\boldsymbol\theta_{[t]}^T{\boldsymbol{\rho}}_{[t]}+\boldsymbol{\eta}_{[t+1]}^T\boldsymbol{\epsilon}_{[t+1]})\Big]\nonumber
\end{align}
where the expectation is with respect to $\epsilon_{t+1}$ and $\boldsymbol{\xi}_{t+1}$ and 
\begin{align}
\bar W_t&(\bar{\textbf{l}}_{t},\boldsymbol{\nu}_{[t]},\text{diag}(\boldsymbol\psi_{[t]}^T\boldsymbol{\nu}_{[t]}+\boldsymbol\eta_{[t+1]}^T\boldsymbol{\xi}_{[t+1]}),\boldsymbol{\rho}_{[t]},\boldsymbol\theta_{[t]}^T{\boldsymbol{\rho}}_{[t]}+\boldsymbol{\eta}_{[t+1]}^T\boldsymbol{\epsilon}_{[t+1]})\nonumber\\[2mm]
=&\max_{\boldsymbol{\pi}_{t+1}\in \Pi_{t+1}(\bar{\textbf{l}}_{t}+\boldsymbol{\nu}_{t},\text{diag}(\boldsymbol\psi_{[t]}^T\boldsymbol{\nu}_{[t]}+\boldsymbol\eta_{[t+1]}^T\boldsymbol{\xi}_{[t+1]})}\Big\{C_{t+1}(\boldsymbol{\pi}_{t+1},\bar{\textbf{l}}_{t}+\boldsymbol{\nu}_{t},\boldsymbol{\nu}_{[t]},\text{diag}(\boldsymbol\psi_{[t]}^T\boldsymbol{\nu}_{[t]}+\boldsymbol\eta_{[t+1]}^T\boldsymbol{\xi}_{[t+1]}),\nonumber\\
&\boldsymbol{\rho}_{[t]},\boldsymbol\theta_{[t]}^T{\boldsymbol{\rho}}_{[t]}+\boldsymbol{\eta}_{[t+1]}^T\boldsymbol{\epsilon}_{[t+1]})+\! \bar V_{t+1}(\bar{\textbf{l}}_{t+1},\boldsymbol{\nu}_{[t]},\text{diag}(\boldsymbol\psi_{[t]}^T\boldsymbol{\nu}_{[t]}+\boldsymbol\eta_{[t+1]}^T\boldsymbol{\xi}_{[t+1]}),\boldsymbol{\rho}_{[t]},\boldsymbol\theta_{[t]}^T{\boldsymbol{\rho}}_{[t]}+\boldsymbol{\eta}_{[t+1]}^T\boldsymbol{\epsilon}_{[t+1]})\Big\}\nonumber
\end{align}
and
\begin{align}
\Pi_{t+1}(\bar{\textbf{l}}_{t}+\boldsymbol{\nu}_{t},&\text{diag}(\boldsymbol\psi_{[t]}^T\boldsymbol{\nu}_{[t]}+\boldsymbol\eta_{[t+1]}^T\boldsymbol{\xi}_{[t+1]}))=\Big\{\boldsymbol{\pi}_{t+1}:\bar{\textbf{l}}_{t+1}=\bar{\textbf{l}}_{t}+\boldsymbol{\nu}_{t}+R\boldsymbol{\pi}_{t+1},\nonumber\\
&\textbf{l}^{min}\leq \bar{\textbf{l}}_{t+1}+\text{diag}(\boldsymbol\psi_{[t]}^T\boldsymbol{\nu}_{[t]}+\boldsymbol\eta_{[t+1]}^T\boldsymbol{\xi}_{[t+1]})\leq \textbf{l}^{max}, \boldsymbol{\pi}^{min}\leq \boldsymbol{\pi}_{t+1}\leq \boldsymbol{\pi}^{max}\Big\}.\nonumber
\end{align}

We now prove that $\bar V_{t}(\bar{\textbf{l}}_{t},\boldsymbol{\nu}_{[t]},\boldsymbol{\rho}_{[t]})$ is concave piece-wise linear in the reservoir levels and inflows. Based on that, we derive an upper bound for the post-decision value function, which is a linear function of the reservoir level and inflows.

\begin{prop} 
For fixed $\boldsymbol{\rho}_{[t]}$, the value function $\bar V_t(\bar{\textbf{l}}_{t},\boldsymbol{\nu}_{[t]},\boldsymbol{\rho}_{[t]})$ is concave and piece-wise linear in $(\bar{\textbf{l}}_{t},\boldsymbol\nu_{[t]})$ for $t=1,...,T$.
\end{prop}
\begin{proof}
Recall that $\bar V_{T}(\bar{\textbf{l}}_{T},\boldsymbol{\nu}_{[T]},\boldsymbol{\rho}_{[T]})=0$ for all $(\bar{\textbf{l}}_{T},\boldsymbol{\nu}_{[T]})$. 

Assume that $\bar V_{t+1}(\bar{\textbf{l}}_{t+1},\boldsymbol{\nu}_{[t+1]},\boldsymbol{\rho}_{[t+1]})$ is concave and piece-wise linear in $(\bar{\textbf{l}}_{t+1},\boldsymbol{\nu}_{[t+1]})$. Since $\bar{\textbf{l}}_{t+1}=\bar{\textbf{l}}_{t}+\boldsymbol{\nu}_{t}+R\boldsymbol{\pi}_{t+1}$ and $\boldsymbol{\nu}_{t+1}=\text{diag}(\boldsymbol\psi_{[t]}^T\boldsymbol{\nu}_{[t]}+\boldsymbol\eta_{[t+1]}^T\boldsymbol{\xi}_{[t+1]})$, $\bar V_{t+1}$ is concave and piece-wise linear in $(\bar{\textbf{l}}_{t},\boldsymbol\nu_{[t]})$ for fixed $\epsilon_{t+1}$ and $\boldsymbol{\xi}_{t+1}$. Thus, $\bar W_t$ is the optimal value of a linear program in which $(\bar{\textbf{l}}_{t},\boldsymbol\nu_{[t]})$ appears in the right-hand side of its constraints and as a result, it is concave and piece-wise linear in $(\bar{\textbf{l}}_{t},\boldsymbol\nu_{[t]})$ for fixed $\epsilon_{t+1}$ and $\boldsymbol{\xi}_{t+1}$. Moreover, the expectation over $\epsilon_{t+1}$ and $\boldsymbol{\xi}_{t+1}$ preserves the concavity and piece-wise linearity. As a result, $\bar V_t(\bar{\textbf{l}}_{t},\boldsymbol{\nu}_{[t]},\boldsymbol{\rho}_{[t]})$ is concave and piece-wise linear in $(\bar{\textbf{l}}_{t},\boldsymbol\nu_{[t]})$.
\end{proof}

We use the concavity and piece-wise linearity to derive a supporting hyperplane to $\bar V_t$, that is, an affine upper bounding function, which coincides with the value function in at least one point. By the supergradient inequality, in the point $(\bar{\textbf{l}}_{t}^{n-1},\boldsymbol\nu_{[t]}^{n-1})$, such hyperplane is given by
\begin{align}
\hat V_t(\bar{\textbf{l}}_{t},\boldsymbol{\nu}_{[t]},\boldsymbol{\rho}_{[t]})=\bar V_{t}(\bar{\textbf{l}}_{t}^{n-1},\boldsymbol{\nu}_{[t]}^{n-1},\boldsymbol{\rho}_{[t]})+\textbf a_{t}^T(\bar{\textbf{l}}_{t}-\bar{\textbf{l}}_{t}^{n-1})+\mathrm{Tr}\big(\textbf b_{[t]}^T(\boldsymbol{\nu}_{[t]}-\boldsymbol{\nu}_{[t]}^{n-1})\big)\nonumber   
\end{align}
with $\hat V_t(\bar{\textbf{l}}_{t},\boldsymbol{\nu}_{[t]},\boldsymbol{\rho}_{[t]})\geq \bar V_t(\bar{\textbf{l}}_{t},\boldsymbol{\nu}_{[t]},\boldsymbol{\rho}_{[t]})$ and $\hat V_t(\bar{\textbf{l}}_{t}^{n-1},\boldsymbol{\nu}_{[t]}^{n-1},\boldsymbol{\rho}_{[t]})=\bar V_t(\bar{\textbf{l}}_{t}^{n-1},\boldsymbol{\nu}_{[t]}^{n-1},\boldsymbol{\rho}_{[t]})$, where $\textbf a_t\in\mathbb{R}^{|J|}$ and $\textbf b_{[t]}\in\mathbb{R}^{t|J|}$ with
\begin{align}
a_{jt}\in\frac{\partial \bar V_t}{\partial l_{jt}}(\bar{\textbf{l}}_{t}^{n-1},\boldsymbol{\nu}_{[t]}^{n-1},\boldsymbol{\rho}_{[t]}), \ j\in J,\nonumber\\    
b_{js}\in\frac{\partial \bar V_t}{\partial \nu_{js}}(\bar{\textbf{l}}_{t}^{n-1},\boldsymbol{\nu}_{[t]}^{n-1},\boldsymbol{\rho}_{[t]}), \ j\in J, s=1,\dots,t,\nonumber    
\end{align}
where ${\partial \bar V_t}/{\partial l_{jt}}$ and ${\partial \bar V_t}/{\partial \nu_{js}}$ are components of a supergradient $\partial V_t$. 

Approximating the post-decision value function by $\hat V_{t+1}(\bar{\textbf{l}}_{t+1},\boldsymbol{\nu}_{[t+1]},\boldsymbol{\rho}_{[t+1]})$, the optimization problem becomes
\begin{align}
\hat W_t(\bar{\textbf{l}}_{t},\boldsymbol{\nu}_{[t+1]},\boldsymbol{\rho}_{[t+1]})=\max_{\boldsymbol{\pi}_{t+1}\in \Pi_{t+1}(\bar{\textbf{l}}_{t}+\boldsymbol{\nu}_{t},\boldsymbol{\nu}_{t+1})}\Big\{C_{t+1}(\boldsymbol{\pi}_{t+1},\bar{\textbf{l}}_{t}+\boldsymbol{\nu}_{t},\boldsymbol{\nu}_{[t+1]},\boldsymbol{\rho}_{[t+1]})\!
+\textbf a_{t+1}^T(\bar{\textbf{l}}_{t}+\boldsymbol{\nu}_{t}+R\boldsymbol{\pi}_{t+1})\Big\}\nonumber\\
+\bar V_{t+1}(\bar{\textbf{l}}_{t+1}^{n-1},\boldsymbol{\nu}_{[t+1]}^{n-1},\boldsymbol{\rho}_{[t+1]})-\textbf a_{t+1}^T\bar{\textbf{l}}_{t+1}^{n-1}+\mathrm{Tr}\big(\textbf b_{[t+1]}^T(\boldsymbol{\nu}_{[t+1]}-\boldsymbol{\nu}_{[t+1]}^{n-1})\big), \nonumber
\end{align}
which involves a linear programming problem and the term $\bar V_{t+1}(\bar{\textbf{l}}_{t+1}^{n-1},\boldsymbol{\nu}_{[t+1]}^{n-1},\boldsymbol{\rho}_{[t+1]})-\textbf a_{t+1}^T\bar{\textbf{l}}_{t+1}^{n-1}$\\ $-\mathrm{Tr}\big(\textbf b_{[t+1]}^T\boldsymbol{\nu}_{[t+1]}^{n-1}\big)$ which is constant with respect to $(\bar{\textbf{l}}_{t},\boldsymbol{\nu}_{[t+1]})$.

\section{The training algorithm}\label{5}

In this section, we propose an offline training algorithm to learn supergradients of the post-decision value function. We let the initial components of the supergradients be $\textbf{a}^0_1,...,\textbf{a}^0_{T-1}$ and $\textbf{b}^0_{[1]},...,\textbf{b}^0_{[T-1]}$ where $\textbf{a}^0_t=(a_{1t}^0,\dots,a_{|J|t}^0)^T$ and $\textbf{b}^0_{[t]}=(\textbf{b}^0_{1[t]},\dots,\textbf{b}^0_{|J|[t]})$. The algorithm iterates over $N$ training samples according to Algorithm \ref{alg}. At iteration $n$, the algorithm uses a sample of inflows $\boldsymbol{\nu}^n_{1},...,\boldsymbol{\nu}^n_{T}$ and prices ${\rho}^n_{1},\dots,{\rho}^n_{T}$. At time $t$, we use the sample values $\boldsymbol{\nu}^n_{t}$ and ${\rho}^n_{t}$ and the current pre-decision reservoir level $\bar{\textbf{l}}_{t}^n$ to sample the optimal value $\hat W_t(\bar{\textbf{l}}_{t},\boldsymbol{\nu}_{[t+1]},\boldsymbol{\rho}_{[t+1]})$, given the approximation of the post-decision value, $\hat V_t(\bar{\textbf{l}}_{t},\boldsymbol{\nu}_{[t]},\boldsymbol{\rho}_{[t]})$. This estimates the post-decision value $\hat V_t^{n}(\bar{\textbf{l}}_{t}^{n},\boldsymbol{\nu}_{[t]}^{n},\boldsymbol{\rho}_{[t]}^{n})$, cf.\ Step 1.(a), and determines the next pre-decision reservoir level at time $t+1$, $\bar{\textbf{l}}_{t+1}^n$, cf. Step 1.(c).

At iteration $n$, we also update the estimate of the post-decision value $\hat V_{t}^n(\bar{\textbf{l}}_{t}^n,\boldsymbol{\nu}^{n}_{[t]},\boldsymbol{\rho}^{n}_{[t]})$. We do this by updating the components of $\mathbf{a}_{t}^n$ and $\mathbf{b}_{[t]}^n$ as follows: 
\begin{align}
a_{jt}^n=(1-\alpha_{n})a_{jt}^{n-1}+\alpha_{n}\big(\hat V_{t}^{n}(\bar{{\textbf{l}}}^n_{t}+\mathbf{e}_j,\boldsymbol{\nu}^n_{[t]},\boldsymbol\rho^n_{[t]})-\hat V_{t}^{n}(\bar{{\textbf{l}}}^n_{t+1},\boldsymbol{\nu}^n_{[t]},\boldsymbol\rho^n_{[t]})\big), \ j\in J\nonumber
\end{align}
and
\begin{align}
b_{js}^n=(1-\alpha_{n})b_{js}^{n-1}+\alpha_{n}\big(\hat V_{t}^{n}(\bar{{\textbf{l}}}^n_{t},\boldsymbol{\nu}^n_{[t]}+\mathbf{e}_{js},\boldsymbol\rho^n_{[t]})-\hat V_{t}^{n}(\bar{{\textbf{l}}}^n_{t},\boldsymbol{\nu}^n_{[t]},\boldsymbol\rho^n_{[t]})\big), \ j\in J, s=1,\dots,t,\nonumber
\end{align}
where $\mathbf{e}_j\in\mathbb{R}^{|J|}$ has $1$ at the $j$th entry and zero otherwise, and $\mathbf{e}_{js}\in\mathbb{R}^{t|J|}$ matrix with $1$ in the $j$th column and the $s$th row and zero otherwise, cf.\ Step 1.(b). 

The determination of $\mathbf{a}_{t}^n$ and $\mathbf{b}_{[t]}^n$, however, requires the solution of $|J|+1$ and $t|J|+1$ optimization problems, respectively. To reduce the number of optimization problems to solve in our computational experiments, we assume that $\mathbf{b}_{[t-1]}^n=0$ in time period $t$ such that $\mathrm{Tr}\big(\textbf{b}_{[t]}^T\boldsymbol\nu_{[t]}\big)=\textbf{b}_{t}^T\boldsymbol\nu_{t}$. As a result, we solve
\begin{align}
&\hat W_t(\bar{\textbf{l}}_{t},\boldsymbol{\nu}_{[t+1]},\boldsymbol{\rho}_{[t+1]})=\max_{\boldsymbol{\pi}_{t+1}\in \Pi_{t+1}(\bar{\textbf{l}}_{t}+\boldsymbol{\nu}_{t},\boldsymbol{\nu}_{t+1})}\Big\{C_{t+1}(\boldsymbol{\pi}_{t+1},\bar{\textbf{l}}_{t}+\boldsymbol{\nu}_{t},\boldsymbol{\nu}_{[t+1]},\boldsymbol{\rho}_{[t+1]})\!
+\textbf a_{t+1}^T(\bar{\textbf{l}}_{t}+\boldsymbol{\nu}_{t}+R\boldsymbol{\pi}_{t+1})\Big\}\nonumber\\
&+\bar V_{t+1}(\bar{\textbf{l}}_{t+1}^{n-1},\boldsymbol{\nu}_{[t+1]}^{n-1},\boldsymbol{\rho}_{[t+1]})-\textbf a_{t+1}^T\bar{\textbf{l}}_{t+1}^{n-1}+\textbf b_{t+1}^T(\boldsymbol{\nu}_{t+1}-\boldsymbol{\nu}_{t+1}^{n-1}), \nonumber
\end{align}
where we update the elements of $\mathbf{b}_{t}^n$ by
\begin{align}
b_{jt}^n=(1-\alpha_{n})b_{jt}^{n-1}+\alpha_{n}\big(\hat V_{t}^{n}(\bar{{\textbf{l}}}^n_{t},\boldsymbol{\nu}^n_{[t]}+\mathbf{e}_{j},\boldsymbol\rho^n_{[t]})-\hat V_{t}^{n}(\bar{{\textbf{l}}}^n_{t},\boldsymbol{\nu}^n_{[t]},\boldsymbol\rho^n_{[t]})\big), \ j\in J.\nonumber
\end{align}
Unfortunately, with this assumption of the supergradient, we are no longer guaranteed an affine upper bound.

\begin{algorithm}
\caption{Offline training}\label{alg}
\begin{enumerate}
\item[0.] Initialize the estimate of the supergradients $a_{jt}^0,b_{jt}^0, j\in J, t=1,\dots,T-1$ and the pre-decision \\ state ${\textbf{l}}_1^n={\textbf{l}}_1,\boldsymbol{\nu}_{1}^n=\boldsymbol{\nu}_{1},\rho_{1}^n=\rho_{1}, n=1,\dots,N$. Let $n=1$.
\item[1.] For $t=0,\dots,T-1:$
\begin{enumerate}
\item (Sample the post-decision value) Solve 
\begin{align}\label{online3}
\hat V_{t}^{n}({\textbf{l}}^n_{t},\boldsymbol{\nu}^n_{[t]},{\boldsymbol{\rho}}_{[t]}^n)=&\max_{\boldsymbol{\pi}_{t+1}\in \Pi_{t+1}(\bar{\textbf{l}}_{t}^n+\boldsymbol{\nu}_{t}^n,\boldsymbol{\nu}_{t+1}^n)}\Big\{C_{t+1}(\boldsymbol{\pi}_{t+1},\bar{\textbf{l}}_{t}^n+\boldsymbol{\nu}_{t}^n,\boldsymbol{\nu}_{[t+1]}^n,\boldsymbol{\rho}_{[t+1]}^n)\!\nonumber\\
&+(\textbf a_{t+1}^{n-1})^T(\bar{\textbf{l}}_{t}^n+\boldsymbol{\nu}_{t}^n+R\boldsymbol{\pi}_{t+1})\Big\}+\bar V_{t+1}(\bar{\textbf{l}}_{t+1}^{n-1},\boldsymbol{\nu}_{[t+1]}^{n-1},\boldsymbol{\rho}_{[t+1]}^{n-1})\nonumber\\[2mm]
&-(\textbf{a}_{t+1}^{n-1})^T\bar{\textbf{l}}_{t+1}^{n-1}+(\textbf{b}_{t+1}^{n-1})^T(\boldsymbol{\nu}_{t+1}^n-\boldsymbol{\nu}_{t+1}^{n-1}), \nonumber
\end{align}
and let $\boldsymbol{\pi}_{t+1}^n$ be an optimal solution. 
\item (Update the estimate of post-decision value) If $t>0$, let 
\begin{align}
a_{jt}^n=(1-\alpha_{n})a_{jt}^{n-1}+\alpha_{n}(\hat V_{t}^{n}({\textbf{l}}^n_{t}+\textbf{e}_j,{\boldsymbol{\nu}}_{[t]}^n,{\boldsymbol{\rho}}_{[t]}^n)-\hat V_{t}^{n}({\textbf{l}}^n_{t}, {\boldsymbol{\nu}}_{[t]}^n,{\boldsymbol{\rho}}_{[t]}^n)), j\in J,\nonumber
\end{align}
and
\begin{align}
b_{jt}^n=(1-\alpha_{n})b_{jt}^{n-1}+\alpha_{n}(\hat V_{t}^{n}({\textbf{l}}^n_{t},\boldsymbol{\nu}^n_{[t]}+\textbf{e}_j,{\boldsymbol{\rho}}_{[t]}^n)-\hat V_{t}^{n}({\textbf{l}}^n_{t}, {\boldsymbol{\nu}}_{[t]}^n,{\boldsymbol{\rho}}_{[t]}^n)), \ j\in J.\nonumber
\end{align}
\item Determine the next pre-decision state ${\textbf{l}}_{t+1}^n$
with 
\begin{align}
{\textbf{l}}_{t+1}^n=\bar{\textbf{l}}_{t}^n+\boldsymbol{\nu}_{t}^n\nonumber
\end{align}
\end{enumerate}
\item[2.] Let $n:=n+1$. If $n\leq N$, go to 1.
\item[3.] Return the estimates of the post-decision values $\hat V_{t}^N, t=0,\dots,T$. 
\end{enumerate}
\end{algorithm}

For online optimization, we use the supergradients of the post-decision value function obtained as well as the constant terms $\bar V_{t+1}(\bar{\textbf{l}}_{t+1}^{N},\boldsymbol{\nu}_{[t+1]}^{N},\boldsymbol{\rho}_{[t+1]}^{N})-(\textbf{a}_{t+1}^{N})^T\bar{\textbf{l}}_{t+1}^{N}-(\textbf{b}_{t+1}^{N})^T\boldsymbol{\nu}_{t+1}^{N}$ from the offline training algorithm. Thus, our online algorithm is identical to Algorithm \ref{alg}, except that we skip Step 1.(b).

\section{Computational results}\label{6}

In this section, we investigate the performance of the proposed ADP approach on a stylized version of a Norwegian hydro-power system \cite{Trine2} and on the more realistic case of the Swiss Kraftwerke Oberhasli AG hydro-power plant \cite{Swiss}. Our focus is convergence and quality of solutions. Source code is run in Matlab, using the YALMIP toolbox with Gurobi solver 8.1.1, on a 16-GB RAM personal computer clocking at 3.1 GHz that are accessible in the online companion \cite{Link}. 

\subsection{Input data}

As a demonstration example, we consider a Norwegian hydro-power system consisting of a cascade of two reservoirs for a time horizon of 48 hours. The upper reservoir is fed by external water inflows from rivers, while the lower reservoir both receives water inflows from rivers and water releases from the upper reservoir. Moreover, each reservoir is connected to a power station in which electricity is generated by releasing the water through a turbine. Data for the technical parameters of the reservoirs are provided in Table \ref{Table I}. For the learning of the value function approximation and for out-of-sample analysis, we generate training and test samples, respectively, of market prices and water inflows from autoregressive processes. Each sample contains prices and inflows for 48 hours. We use the following models (\ref{ARMA prices}) and (\ref{ARMA inflows}) from \cite{Trine} for generating the samples:
\begin{equation}\allowdisplaybreaks
(1-\theta_1B)(1-B)(1-B^{24})(1-B^{168})\rho_t=(1-\eta_1B)(1-\eta_{24}B)(1-\eta_{168}B^{168})\epsilon_t, \quad t\in \mathbb{Z},
\label{ARMA prices}
\end{equation}
\begin{equation}\allowdisplaybreaks
(1-\psi_1^jB)(1-B)\nu_{jt}=(1-\phi_1^jB-\phi_2^jB^{2})(1-\phi_{41}^jB^{41})\xi_{jt}, \quad j=1,2, t\in \mathbb{Z},
\label{ARMA inflows}
\end{equation}
where $B$ is the backshift operator, e.g.\ $B^i\rho_t=\rho_{t-i}$.

We assume that the stochastic processes of prices and inflows are uncorrelated. The parameter estimates of \eqref{ARMA prices} are $\theta_1=0.6874, \eta_1=0.9234, \eta_{24}=0.8502, \eta_{168}=0.9665$, and those of \eqref{ARMA inflows} are $\psi_1^1=0.9899, \phi_1^1=1.3156, \phi_2^1=-0.3504, \phi_{41}^1=0.8424$ for the upstream reservoir and $\psi_1^2=0.9775, \phi_1^2=1.4442, \phi_2^2=-0.5509, \phi_{41}^2=0.8304$ for the downstream reservoir. The random variables $\epsilon_{t}$ in (\ref{ARMA prices}) and $\xi_{1t}$ and $\xi_{2t}$ in (\ref{ARMA inflows}) are independent and identically Normally distributed over time with zero means and standard deviations of 0.2369, 0.6549 and 0.1646, respectively. The correlation between $\xi_{1t}$ and $\xi_{2t}$ is 0.0417. We consider the initial value of prices and inflows equal to $20 \ \$/MWh$ and $50 \ 10^3m^3/h$, respectively. 

\begin{table}[]
	\centering
\caption{Data for upper and lower reservoirs.}
\label{Table I}
\begin{tabular}{c c c c c c}
\hline \hline
\multirow{3}{*}{ {Reservoirs}} & Max reservoir & Max reservoir & Min reservoir& Initial & Rate of discharge \\ 
  & discharge  &capacity   &  capacity& reservoir level & to generation\\ 
  &  ($10^3m^3/h$)  &  ($10^3m^3/h$)  &  ($10^3m^3/h$)&  ($10^3m^3/h$)& ($MWh/10^3m^3$) \\\hline
Upper reservoir   & 57.96 &  1130  & 113 & 124.3 & 0.1101 \\
Lower reservoir & 121.36  & 1000  & 100 & 110 & 0.5051 \\
 \hline \hline  
\end{tabular}
\end{table}

\subsection{Convergence and running time}

 \begin{figure}[htb]
	\centering
	\includegraphics[width=0.6\columnwidth]{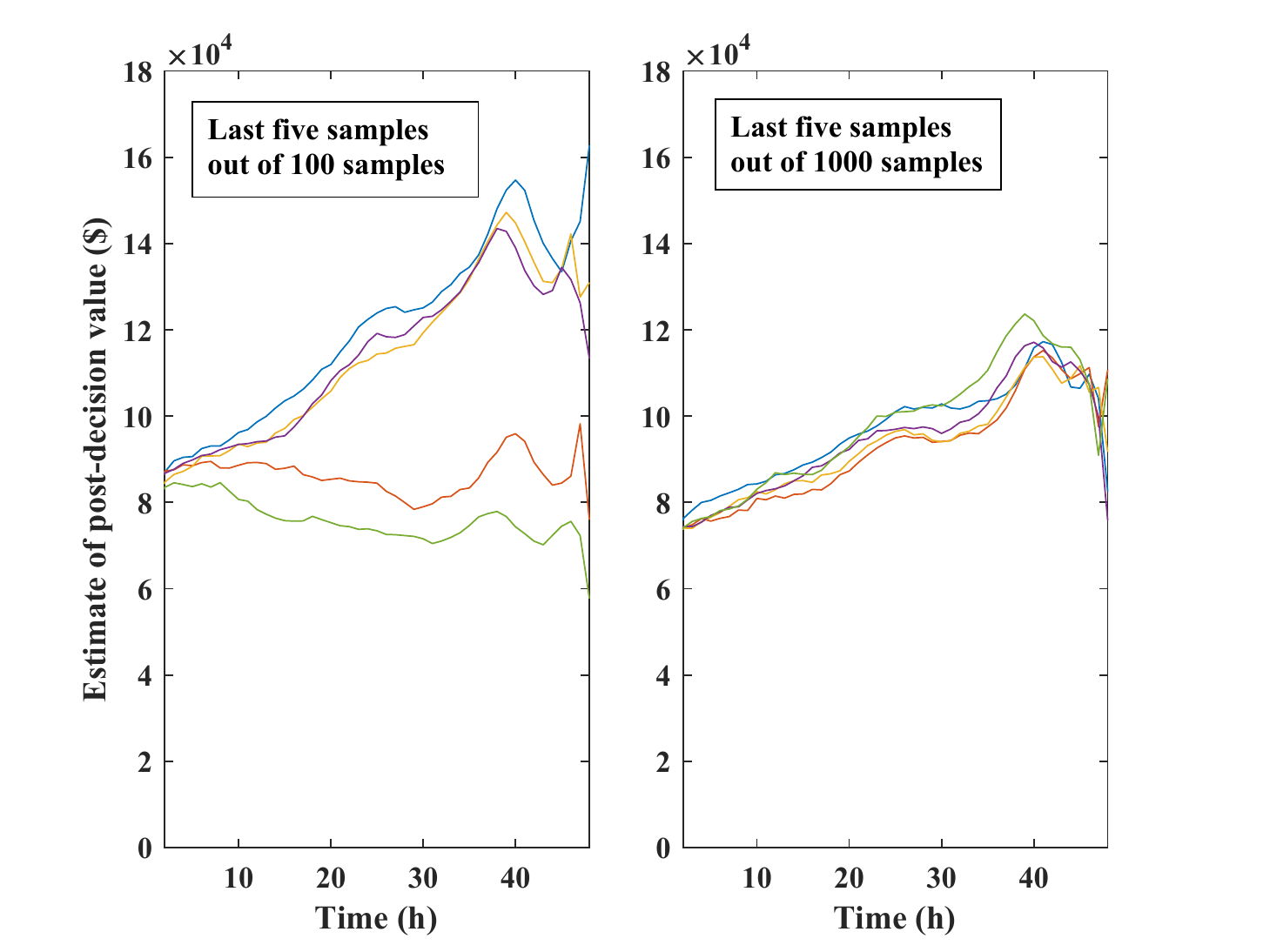}
	\caption{Estimate of post-decision values over 48 hours for the last five samples out of 100 and 1000 samples.}
	\label{convergence}
\end{figure}  

 \begin{figure}[!htb]
	\centering
	\includegraphics[width=0.5\columnwidth]{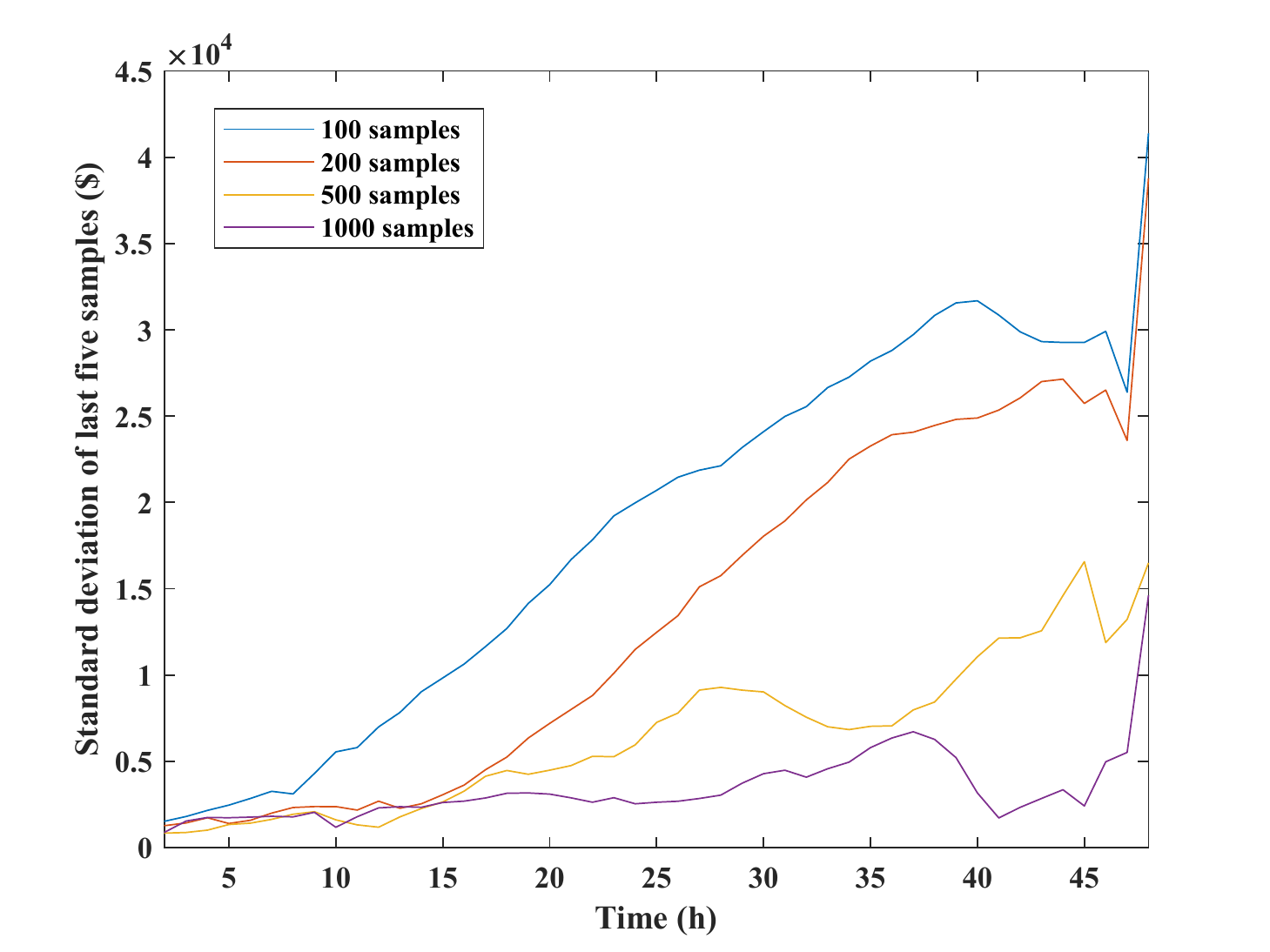}
	\caption{Standard deviation of last five samples out of 100, 200, 500, and 1000 samples.}
	\label{std}
\end{figure}

\begin{figure}[!htb]
	\centering
	\includegraphics[width=0.5\columnwidth]{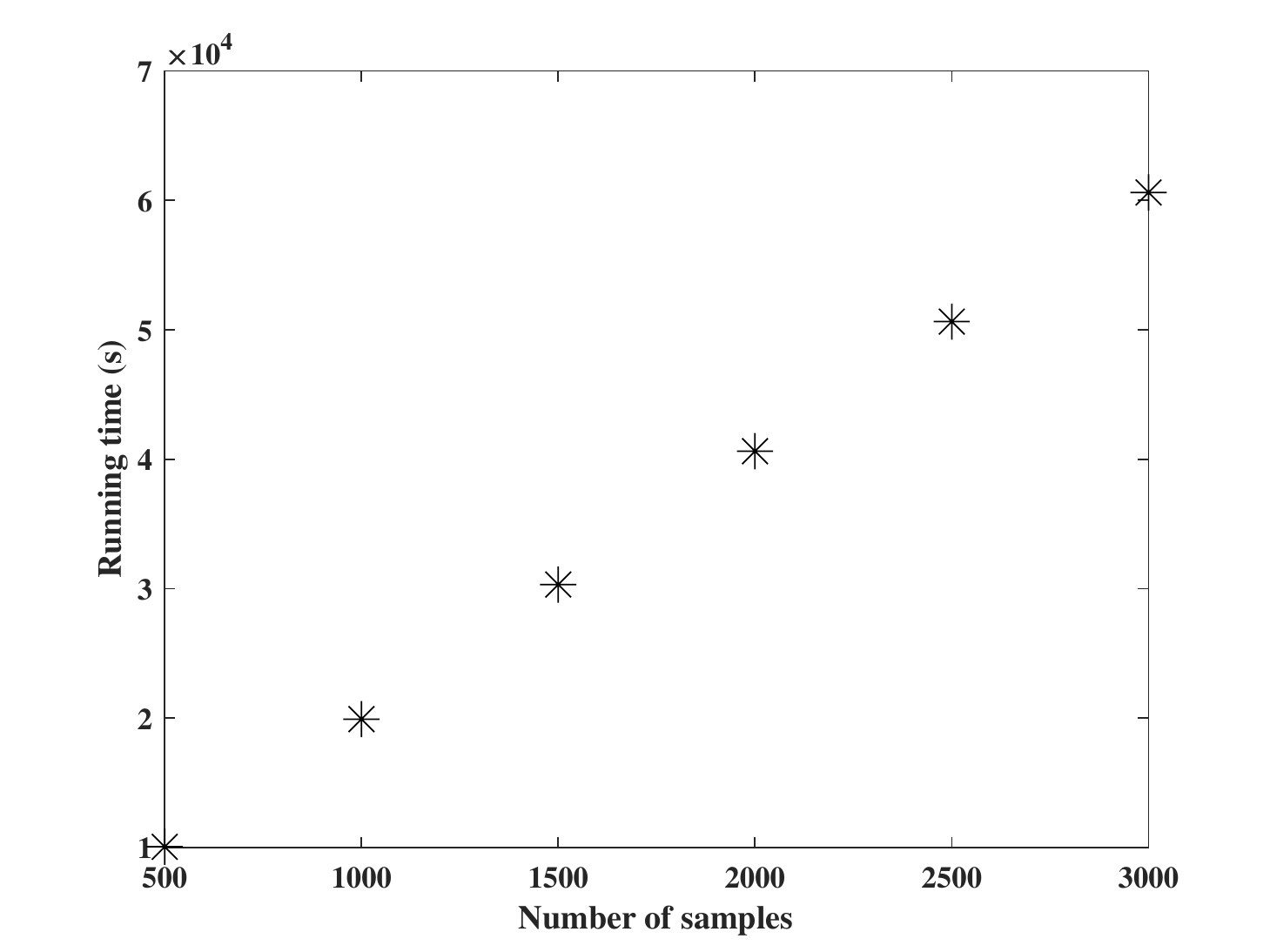}
	\caption{Running time for 500, 1000, 1500, 2000, 2500, and 3000 samples.}
	\label{runningtime}
\end{figure}
\begin{figure}[htb]
	\centering
	\includegraphics[height=7cm, width=0.7\columnwidth]{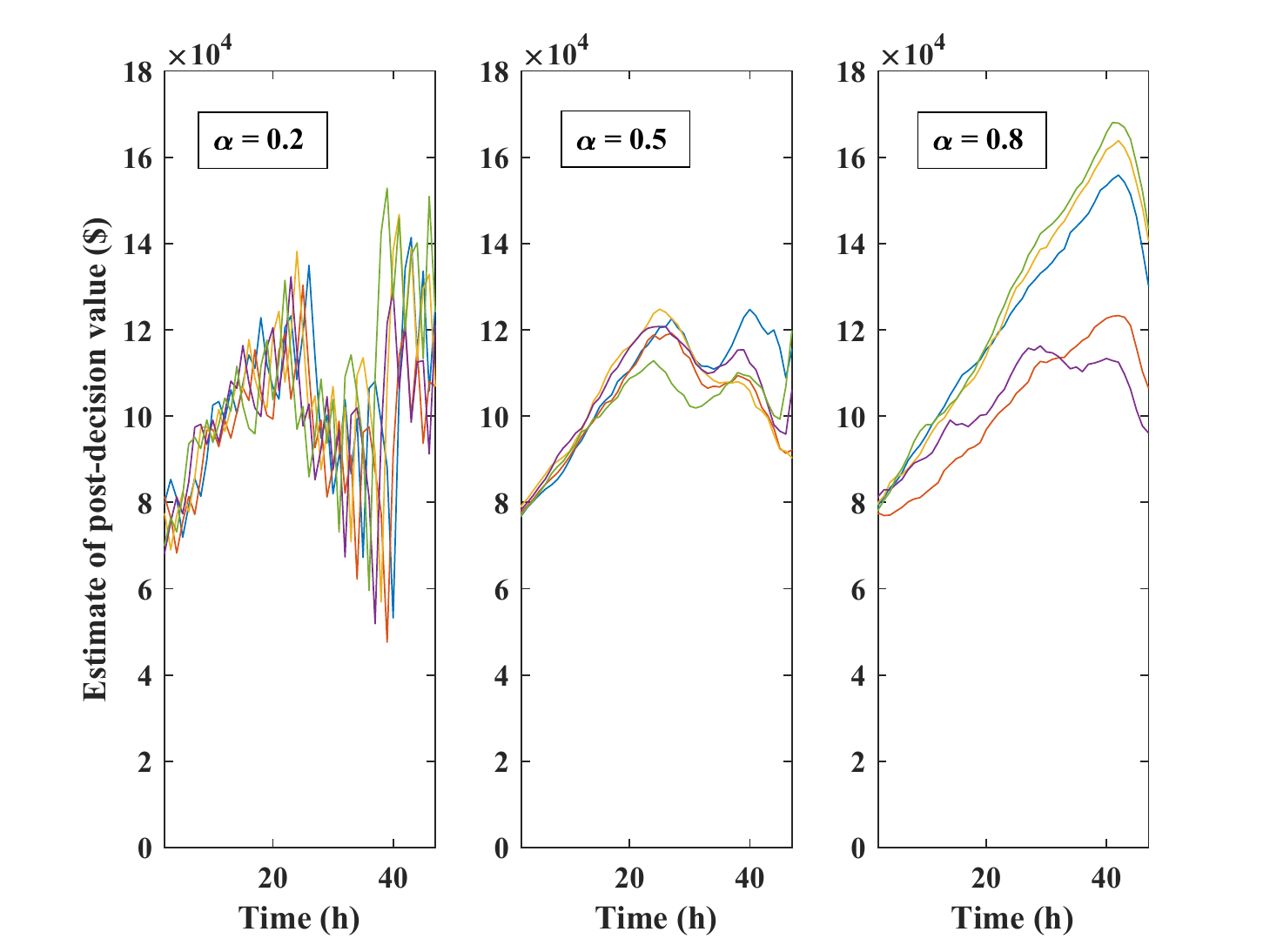}
	\caption{Convergence for different values of $\alpha$.}
	\label{alpha_convergence}
\end{figure}
 
To provide numerical evidence of convergence, we compare the performance of the algorithm for a varying number of training samples. The initial learning rate $\alpha$ is set to $0.5$. Fig. \ref{convergence} depicts the estimated post-decision values, $\hat V_t, t=0,\dots,48$, for the last five samples out of 100 and 1000 samples, respectively. Moreover, Fig.\ \ref{std} shows the standard deviation of the last five samples out of 100, 200, 500, and 1000 samples. As it can be observed, the estimates vary much less, the larger the number of samples. In fact, the average standard deviation decreases from 16.87\% of the mean (100 samples) to 3.50\% of the mean (1000 samples). Clearly, a larger number of samples results in a better convergence. This is, however, at the expense of higher computational time. Fig.\ \ref{runningtime} depicts the running time of the ADP algorithm as a function of the sample size. As expected, the running time is seen to increase linearly with the number of samples. With $500$ samples, the running time is approx.\ an hour. To  investigate how convergence depends on the learning rate $\alpha$, we finally vary this parameter. Fig.\ \ref{alpha_convergence} illustrates the importance of parameter tuning.

We confirm convergence towards the exact optimal value in a deterministic setting. The deterministic reservoir problem is equivalent to the following linear programming (LP) formulation:
\begin{subequations}
\label{lp}
\begin{align}
\max\ \ & \sum_{t=1}^T\rho_t \textbf{g}^T\boldsymbol{\pi}_{t}+\rho_{T+1}\textbf{g}^T \textbf{l}_{T+1}\\
\text{st}\ \ & \textbf{l}_{t+1}=\textbf{l}_{t}+R\boldsymbol{\pi}_{t}+\boldsymbol{\nu}_{t}, & t=1,\dots,T\nonumber\\[1mm]
&\textbf{l}^{min}\leq \textbf{l}_{t+1}\leq \textbf{l}^{max}, &  t=1,\dots,T\\[1mm]
&\boldsymbol\pi^{min}\leq \boldsymbol\pi_{t}\leq \boldsymbol\pi^{max}, & t=1,\dots,T.
\end{align}
\end{subequations}
The formulation is exact in the sense that it does not involve any approximation. Using the same sample of prices and inflows, we both solve the LP problem and run the ADP algorithm for 200 iterations. We compare the exact and estimated post-decision value at time $t=0$, $\hat V_0$, from LP and ADP, respectively. By repeating the procedure for ten samples, we obtain an average inaccuracy of $1\%$. With the computational challenges of stochastic programming, a comparison between an exact LP formulation and the ADP approach is feasible only for the deterministic problem. 

\subsection{The quality of solutions}

For the analysis of the solutions, we first run the offline ADP algorithm to train the supergradients of the post-decision value functions, using 1,000 training samples. Next, we fix $\textbf{a}_t$ and $\textbf{b}_t$ to the values obtained from the last iteration of the offline algorithm and run the online algorithm to obtain an optimal solution $(\boldsymbol{\pi}_1^*,...,\boldsymbol{\pi}_T^*)$ for each sample. 

Figs. \ref{900} and \ref{927} illustrate the discharging and reservoir levels for an arbitrary sample of prices and external inflows. As expected, the reservoirs start discharging as the price increases. The upper reservoir starts releasing water at lower prices than the lower reservoir, since upstream water releases can be used to produce power in the upper power station, but continue downstream and can likewise be used to produce in the lower power station. At high prices, the lower reservoir naturally releases more water than the upper reservoir. Due to external inflows, the reservoir levels increase when no discharging occurs. In spite of inflows, however, the reservoir levels decrease when discharging. For the lower reservoir, inflows consist of external inflows and the water from the upper reservoir. Thus, its reservoir level will increase at a higher speed when the upper reservoir discharges. Since reservoir discharges are not only affected by prices but also by external inflows, the upper reservoir cannot keep discharging at its maximum level as the inflow decreases, since there is insufficient water in the reservoir. Consequently, this reservoir hits its minimum reservoir level.

\begin{figure}[!htb]
	\centering
	\includegraphics[width=0.8\columnwidth]{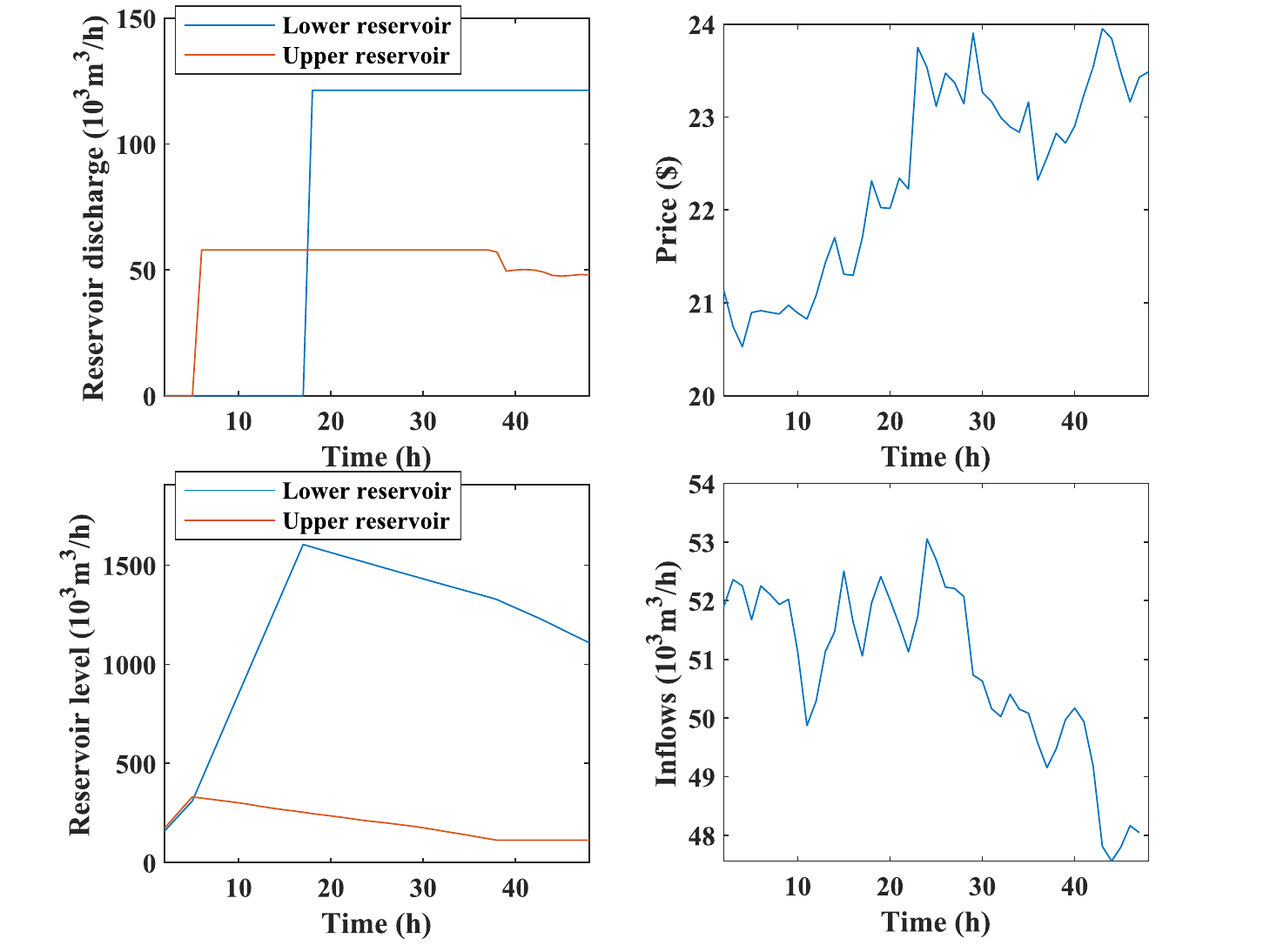}
	\caption{Reservoir discharge and reservoir level with respect to prices and external inflows for sample number 900.}
	\label{900}
\end{figure}

\begin{figure}[!htb]
	\centering
	\includegraphics[width=0.8\columnwidth]{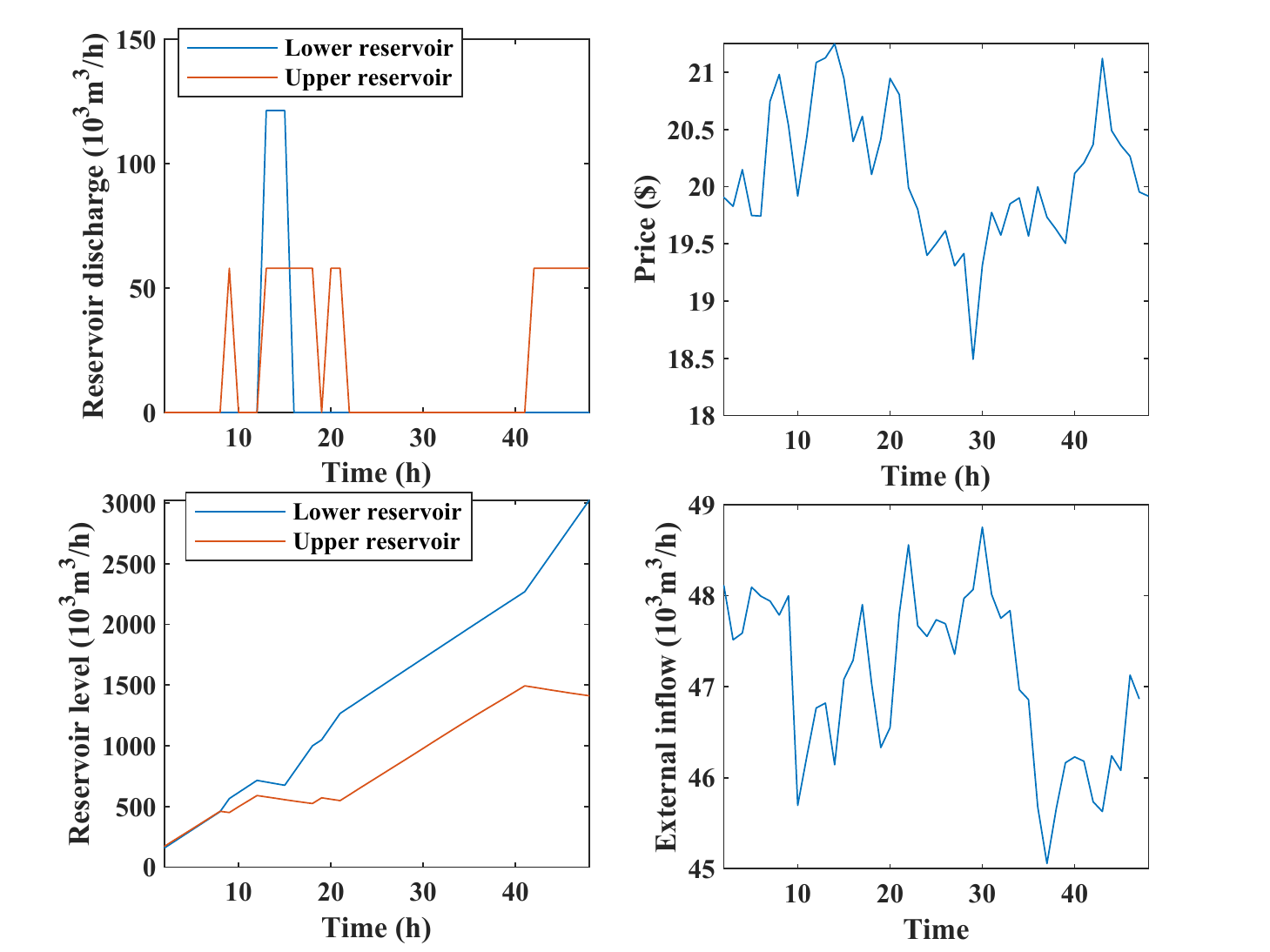}
	\caption{Reservoir discharge and reservoir level with respect to prices and external inflows for sample number 927.}
	\label{927}
\end{figure}

To quantify in-sample and out-of-sample performance of the solutions, we compute the average of the post-decision value at $t=0$ over all samples. We carry out this calculation for samples of both training and test data sets, to derive so-called the in-sample and out-of-sample values, respectively, see the second and third column of Table \ref{Table II}. The difference between the in-sample and out-of-sample values is on average 0.05\%, indicating stability of the profit estimation.

To further assess the quality of solutions, we compare the estimate from the ADP algorithm with the so-called wait-and-see value. A deterministic wait-and-see problem assumes perfect information is available and represented by a sample. The wait-and-see value is the expected value of having such perfect information, i.e.\ the average optimal value over all samples or wait-and-see problems. We use the samples of both test and training data set to compute the wait-and-see value, and solve the deterministic linear program (\ref{lp}) for each sample. Using the test data set, the value is equal to $7.7072 \times 10^4$ whereas this value varies for different number of training samples, e.g. it is $7.7052\times 10^4$ when using 5000 samples. It is observed that the difference between ADP and wait-and-see values is less than $2\%$.  



\begin{table}[h!]
\centering
\caption{Estimates of the post-decision value obtained from the ADP approach.}
\label{Table II}
\begin{tabular}{c c c c c}
\hline \hline
Number of samples & In-sample ($\times10^4\$$) & Out-of-sample ($\times10^4\$$) & Diff. (\%)\\\hline
200 & 8.1077  & 8.1049 & 0.034 \\
1000 & 7.8553  & 7.8540 & 0.165\\
2000 & 7.7926  & 7.7870 & 0.072 \\
3000 & 7.6529  &  7.6506 & 0.028\\
4000 & 7.5740  &  7.5730 & 0.012\\ 
5000 & 7.6536  &  7.6535 & 0.001\\
\hline \hline 
\end{tabular}
\end{table}

\begin{table}[h!]
\centering
\caption{Estimate of the post-decision value obtained from the ADP approach in case E.}
\label{Table II Case E}
\begin{tabular}{c c c c c c}
\hline \hline
Number of samples & In-sample ($\times10^4\$$) & Diff.($\times10^4\$$) & Out-of-sample ($\times10^4\$$) &  Diff.($\times10^4\$$) \\\hline
200 & 6.5672  &1.5405& 6.1643 & 1.9406 \\
100   &6.3932 &1.4621& 6.2399 & 1.6141\\
2000  &6.3853  &1.4073& 6.2415 & 1.5455\\
3000   &5.8910 &1.7619&  5.8680 & 1.7826\\
4000  &5.7555  &1.8185& 5.7352 & 1.8378\\
5000  &5.9697  &1.6839& 5.9544 & 1.6991\\
\hline \hline 
\end{tabular}
\end{table}

\subsection{Approximation}

To demonstrate the fact that profit accurate estimation not only depends on the current reservoir level but also on future inflows, we compare two cases wherein the term $\textbf{b}_{t}^T\boldsymbol\nu_{t}$ in the offline and online algorithms is excluded (E) and included (I), respectively. In the deterministic problem, i.e.\ considering only one sample, the $\hat V_0$ obtained from the ADP approach is $5.7580 \times10^4\$$ in case E while it is $8.2257\times10^4\$$  in case I. Similarly, for the stochastic problem, Table \ref{Table II Case E} demonstrates that profit estimation on average improve by 20\% when including inflows in the linear approximation. This shows how crucial it is to include future inflows in the estimation of the post-decision value.

\color{black}

\section{Case study of a reservoir network} \label{7}

To demonstrate the applicability of ADP to more complex systems, we proceed with a case study of the Swiss Kraftwerke Oberhasli AG; a hydro-power plant including multiple reservoirs connected in a network architecture \cite{Swiss}.

\subsection{Modeling}

The network architecture includes a number of reservoirs, possibly equipped with a power station and/or a pump. Water can be released to supply electricity at the wholesale market price or it can be pumped in the opposite direction by consuming power purchased from the market at the same price. Water releases from upstream reservoirs contribute to downstream inflows and pumping from downstream power stations results in upstream inflows. We extend the notation accordingly. 

Let $\boldsymbol{\pi}^d_{t}=(\pi^d_{1t},\dots,\pi^d_{|J|t})^T\in \mathbb{R}^{|J|}$ and $\boldsymbol{\pi}^c_{t}=(\pi^c_{1t},\dots,\pi^c_{|J|t})^T\in \mathbb{R}^{|J|}$ represent the charges and dischargess of the reservoirs during time period $t$ and $\textbf{f}_{t} =(f_{1t},\dots,f_{|\Gamma|t})^T\in \mathbb{R}^{|\Gamma|}$ be the water flow in the tunnels, where $\Gamma$ is the set of tunnels, i.e.\ interconnected pairs of reservoirs. The set of feasible decisions is defined by  
\begin{subequations}
\begin{align}
\Pi_{t+1}(\bar{\textbf{l}}_{t}+\boldsymbol{\nu}_{t},\boldsymbol{\nu}_{t+1})=\Big\{(\boldsymbol{\pi}^d_{t+1},\boldsymbol{\pi}^c_{t+1}):\ &\bar{\textbf{l}}_{t+1}=\bar{\textbf{l}}_{t}+\boldsymbol{\nu}_{t}-\boldsymbol{\pi}^d_{t+1}+\eta\boldsymbol{\pi}^c_{t+1},\label{c1} \\ &\textbf{l}^{min}\leq \bar{\textbf{l}}_{t+1}+\boldsymbol{\nu}_{t+1}\leq \textbf{l}^{max},\label{c2}\\ &\boldsymbol\pi^{min}\leq \boldsymbol{\pi}^d_{t+1},\boldsymbol{\pi}^c_{t+1}\leq \boldsymbol\pi^{max},\label{c3}\\[1mm]
&0\leq \textbf{f}_{t+1}\leq \textbf{f}^{max},\label{c4}\\
&\boldsymbol\pi_{t+1}^d=R^d\ \textbf{f}_{t+1} , \boldsymbol\pi_{t+1}^c=R^c\ \textbf{f}_{t+1} \label{c5}\Big\}
\end{align}
\end{subequations}
where the matrices $R^d\in \mathbb{R}^{|J|}\times\mathbb{R}^{|\Gamma|}$ and $R^c\in \mathbb{R}^{|J|}\times\mathbb{R}^{|\Gamma|}$ illustrate which reservoirs can be charged and discharged through which tunnels, i.e\ for $\gamma\in\Gamma$ and $j \in J$, $R^d_{j\gamma}=1$ if $\gamma=(j,k), j\in J^-(k)$, $R^c_{j\gamma}=1$ if $\gamma=(k,j), k\in J^+(j)$, and $R^d_{i\gamma}=R^d_{i\gamma}=0$, otherwise. $J^-(j)$ and $J^+(j)$ denote the reservoirs immediately upstream and downstream from reservoir $j$, respectively. In addition to the reservoir balances \eqref{c1}, capacity constraints \eqref{c2} and charging and discharging limits \eqref{c3}, we include the capacity limits of the tunnels \eqref{c4} and the connection of reservoirs and tunnels \eqref{c5}. In \eqref{c1}, $\eta$ denotes pumping deficiency. 

When upstream reservoir $j$ releases a water flow of $f_{(j,k)t}$ to reservoir $k$ through tunnel $(j,k)$ at time $t$, the turbines generate electricity with a conversion rate of $g_{(j,k)}$ to be sold at market price $\rho_t$. In contrast, if downstream reservoir $j$ pumps water $f_{(j,k)t}$ to $k$ through tunnel $(j,k)$, it consumes electricity with conversion rate $g_{(j,k)}$ which is bought from the market at price $\rho_t$. Thus, at each stage $t$, the post-decision value satisfies
\begin{align}
\label{connected_obj}
\bar V_t(\bar{\textbf{l}}_{t},\boldsymbol{\nu}_{[t]},\boldsymbol{\rho}_{[t]})=&\mathbb{E}\Big[\max_{(\boldsymbol{\pi}^d_{t+1},\boldsymbol{\pi}^c_{t+1})\in \Pi_{t+1}(\bar{\textbf{l}}_{t},\boldsymbol{\nu}_{t},\boldsymbol{\nu}_{t+1})}\Big\{\rho_{t+1}\textbf{g}^T\textbf{f}_{t+1}+\bar V_{t+1}(\bar{\textbf{l}}_{t+1},\boldsymbol{\nu}_{[t+1]},\boldsymbol{\rho}_{[t+1]})\Big\}\Big|\boldsymbol{\nu}_{[t]},\boldsymbol{\rho}_{[t]}\Big],
\end{align}
where
\begin{equation}
  \mathbf{g}= 
  \begin{cases}
      g_{(j,k)} &\quad k\in J^+(j),\\
      -g_{(j,k)} &\quad j\in J^-(k).
    \end{cases}\nonumber 
  \end{equation}


\subsubsection{Results}

The larger test system here is an adapted version of the Swiss system of the Kraftwerke Oberhasli AG hydro-power plant, including six reservoirs, equipped with power stations and/or a pumps, and connected by five tunnels, as shown in Fig. \ref{Swiss Hydro}. Technical data for reservoirs and tunnels is provided in Tables \ref{Table III} and \ref{Table V}. We consider the same capacity of downstream and upstream tunnels. For both the releasing and pumping processes, we consider a generator and a pump with the same conversion rate and capacity limit. We set $\eta=0.6$. To generate training and test samples, we use the same ARMA models as in \eqref{ARMA inflows}-\eqref{ARMA prices}. Yet, we scale the inflows according to the capacity of reservoirs. 

\begin{figure*} [t]
\centering
\begin{minipage}{.5\textwidth}
  \centering
  	\scalebox{0.9}{
	\includegraphics[width=1\columnwidth]{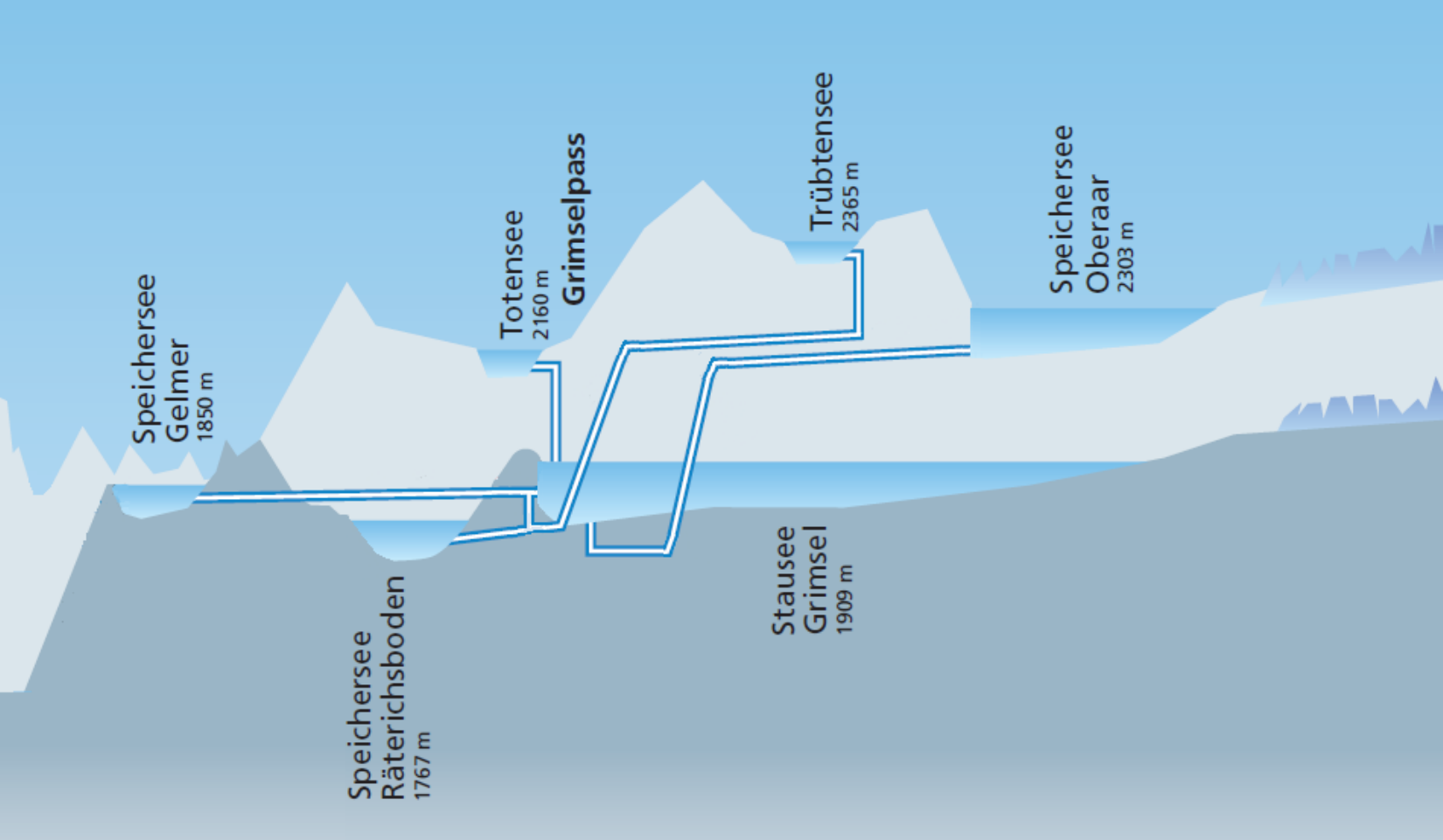}
	}
\end{minipage}%
\begin{minipage}{.5\textwidth}
  \centering
    	\scalebox{0.9}{
	\includegraphics[width=0.89\columnwidth]{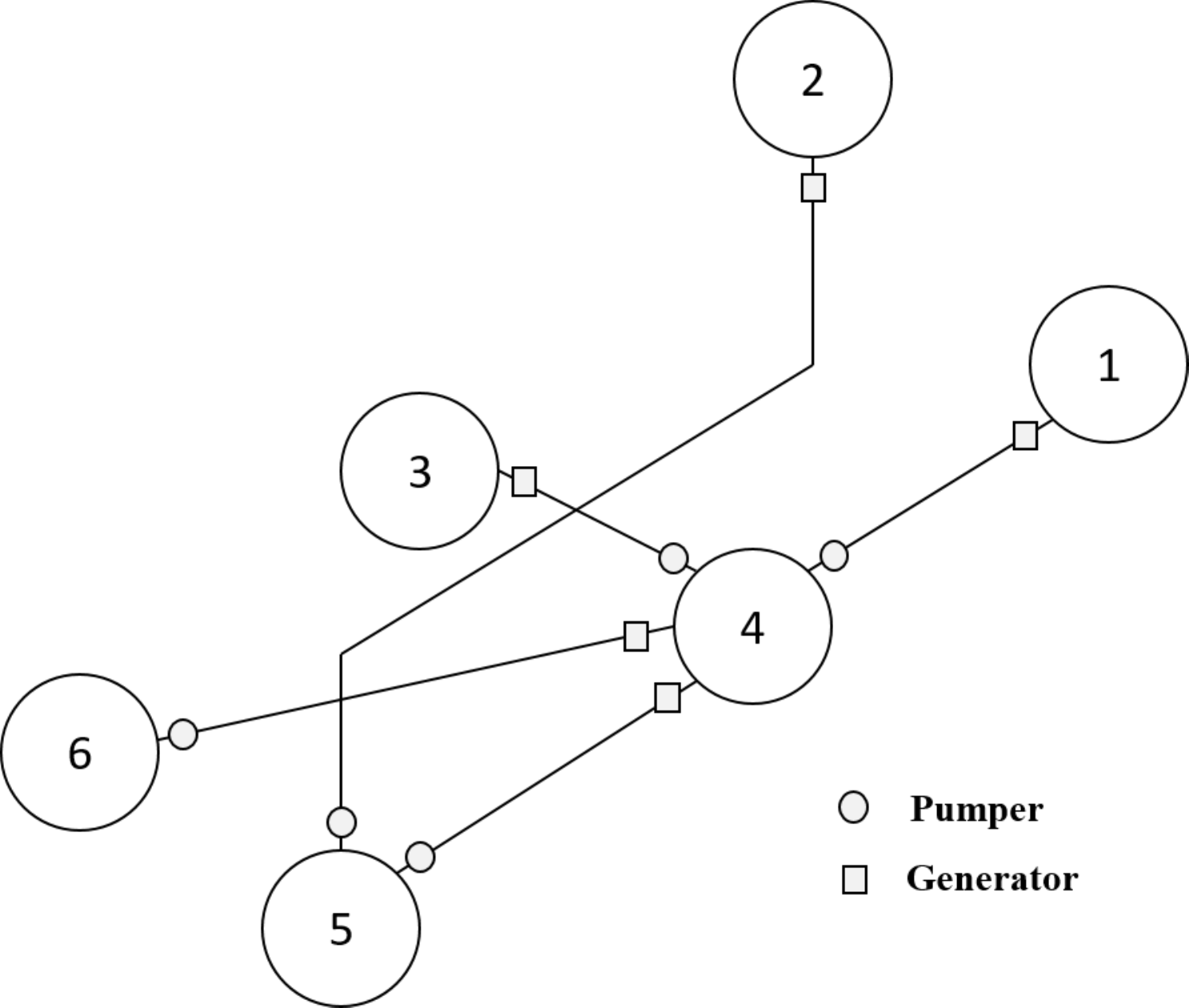}
	}
\end{minipage}
    \caption{The Kraftwerke Oberhasli AG (KWO) hydro power plant. }
\label{Swiss Hydro}
\end{figure*}

\begin{table}
\caption{Data for reservoirs.}
\label{Table III}
\centering
\begin{tabular}{c c c c c}
\hline \hline
\multirow{3}{*}{ {Reservoirs}} & Max reservoir & Max reservoir & Min reservoir& Initial \\ 
  & discharge  &capacity   &  capacity& reservoir level \\ 
  &  ($10^3m^3/h$)  &  ($10^3m^3$)  &  ($10^3m^3$)&  ($10^3m^3$)\\\hline
1   & 2.39 &  65.07  & 6.50 & 7.15 \\
2 & 0.11  & 1.14  & 0.11 & 0.12 \\
3 & 0.22  & 2.28  & 0.22 & 0.25 \\
4 & 3.02  & 107.30  & 10.73 & 11.80 \\
5 & 0.23  & 28.53  & 2.85 & 3.13 \\
6 & 1.10  & 1.14  & 0.11 & 0.12  \\
 \hline \hline  
\end{tabular}
\end{table}

\begin{table}[h!]
\caption{Data for tunnels.}
\label{Table V}
\centering
\begin{tabular}{c c c c c c}
\hline \hline
{Tunnels} & (1,4),(4,1) & (2,5),(5,2) & (3,4),(4,3) & (4,5),(5,4) & (4,6),(6,4) \\\hline
Conversion rate ($MWh/10^3m^3$) & 0.1 &  0.04  & 0.03 & 0.1 & 0.03 \\
Maximum capacity ($10^3m^3/h$) & 2.52  & 3.10  & 0.22 & 3.61 & 2.52 \\
 \hline \hline  
\end{tabular}
\end{table}

We run the proposed algorithm for 1000 training samples and a time horizon of 48 hours. The parameter $\alpha$ is set to $0.5$. Consistent with \eqref{connected_obj}, the estimate of the post-decision value in Step 1. (a) of Algorithm \ref{alg} is replaced by   
\begin{align}\label{online3}
\hat V_{t}^{n}({\textbf{l}}^n_{t},\boldsymbol{\nu}^n_{[t]},{\boldsymbol{\rho}}_{[t]}^n)=&\max_{(\boldsymbol{\pi}_{t+1}^d,\boldsymbol{\pi}_{t+1}^c)\in \Pi_{t+1}(\bar{\textbf{l}}_{t}^n+\boldsymbol{\nu}_{t}^n,\boldsymbol{\nu}_{t+1}^n)}\Big\{C_{t+1}(\mathbf{f}_{t+1},\bar{\textbf{l}}_{t}^n+\boldsymbol{\nu}_{t}^n,\boldsymbol{\nu}_{[t+1]}^n,\boldsymbol{\rho}_{[t+1]}^n)\!\nonumber\\
&+(\textbf a_{t+1}^{n-1})^T(\bar{\textbf{l}}_{t}^n+\boldsymbol{\nu}_{t}^n-\boldsymbol{\pi}_{t+1}^d+\eta\boldsymbol{\pi}_{t+1}^c)\Big\}+\bar V_{t+1}(\bar{\textbf{l}}_{t+1}^{n-1},\boldsymbol{\nu}_{[t+1]}^{n-1},\boldsymbol{\rho}_{[t+1]}^{n-1})\nonumber\\[2mm]
&-\textbf (\mathbf{a}_{t+1}^{n-1})^T\bar{\textbf{l}}_{t+1}^{n-1}+\textbf (\mathbf{b}_{t+1}^{n-1})^T(\boldsymbol{\nu}_{t+1}^n-\boldsymbol{\nu}_{t+1}^{n-1}). \nonumber
\end{align}
\color{black}
The other steps remain the same. 

The estimates of the post-decision value for last five samples out of 100 and 1000 samples, respectively, are illustrated in Fig.\ \ref{convergence_connected}. As for the demonstration example, the larger number of samples results in a better convergence. By increasing the number of training samples from 100 to 1000, the average standard deviation of last five samples decreases from $8.85\%$ to $3.95\%$. A running time of an hour allows for the use of approx. 300 samples.

 \begin{figure}[!htb]
	\centering
	\includegraphics[width=0.6\columnwidth]{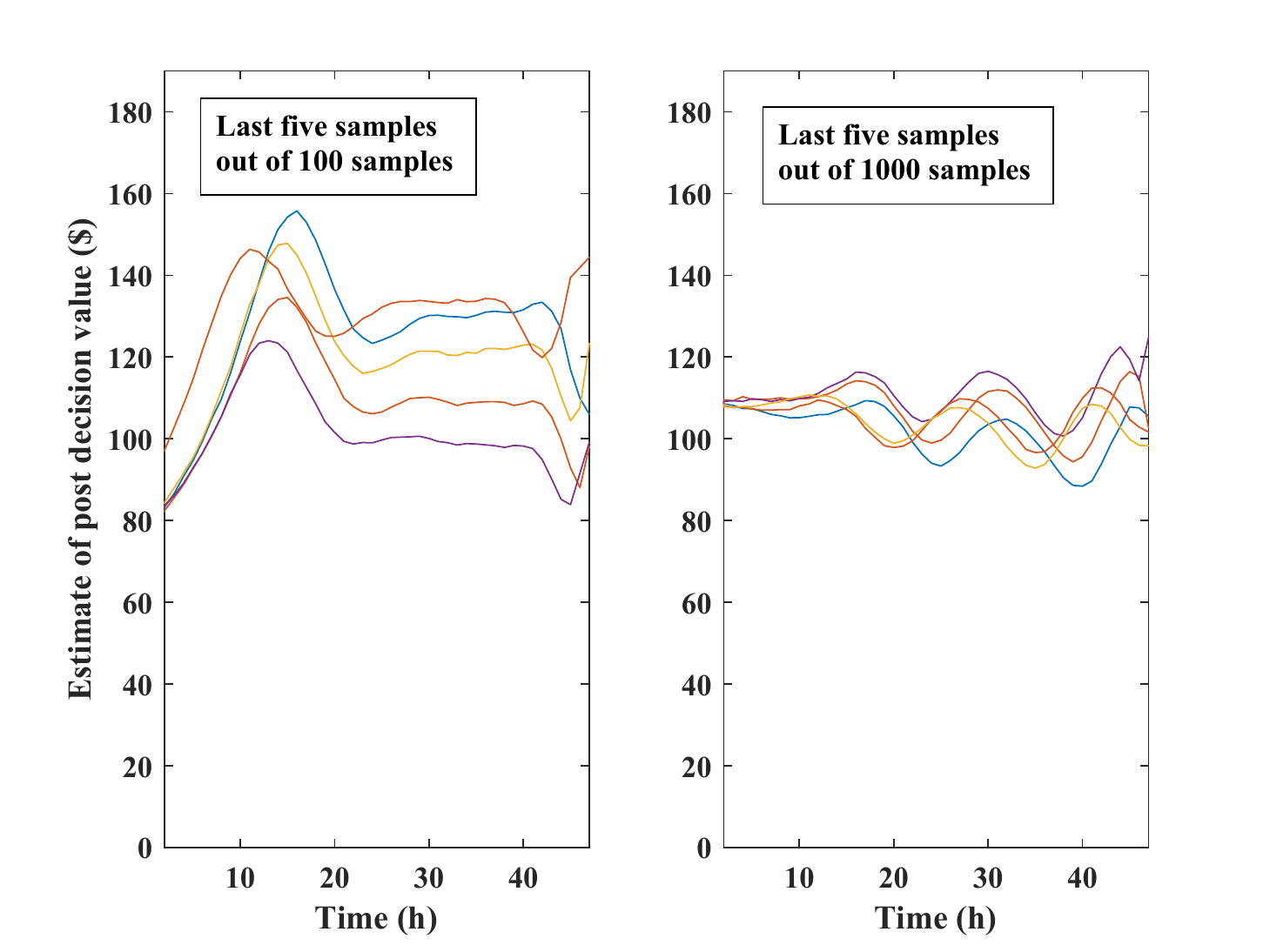}
	\caption{Estimate of post-decision values over 48 hours for the last five samples out of 100 and 1000 samples.}
	\label{convergence_connected}
\end{figure}

\begin{table}[h!]
\centering
\caption{Estimate of post-decision value obtained from the ADP approach.}
\label{Table VI}
\begin{tabular}{c c c c }
\hline \hline
Number of samples & In-sample (\$)  & Out-of-sample (\$)  \\\hline
100 & 114.5570 & 113.6344\\
1000  & 92.5732 & 92.3699 \\
\hline \hline 
\end{tabular}
\end{table}

\begin{table}[h!]
\centering
\caption{Estimate of post-decision value obtained from the ADP approach in case E.}
\label{Table VI Case E}
\begin{tabular}{c c c c }
\hline \hline
Number of samples & In-sample (\$)  & Out-of-sample (\$)  \\\hline
100 & 89.0566 & 88.7890 \\
1000  & 77.7650 & 76.8870\\
\hline \hline 
\end{tabular}
\end{table}
 Table \ref{Table VI} lists the in-sample and out-of-sample post-decision values for different numbers of samples. Even for the realistically sized case, the difference between in-sample and out-of-sample values remains less than 2\%. In another analysis, we compare the ADP value to the wait-and-see value, which results in an average difference of $8.64\%$.  As for the demonstration example, we finally consider two cases of E and I wherein the term $\textbf{b}_{t}^T\boldsymbol\nu_{t}$ in the offline and online algorithms is excluded and included, respectively. By comparing Tables \ref{Table VI} and \ref{Table VI Case E} it can be observed that the profit estimation improve on average by 18\% in case I.

\section{Conclusion} \label{8}
This paper proposes an approximate dynamic programming approach to estimate future profits of connected hydro reservoirs. To overcome dimensionality issues, we use the so-called post-decision state and a linear approximation architecture. We prove that when the time series of prices and inflows follow an autoregressive process, our approximation provides an upper bound on future profits. 

We assess the performance of our proposed model in terms of convergence and quality of solutions for a stylized systems of reservoirs in cascade and for a more realistic network of connected reservoirs. In both cases, we obtain convergence of the profit value in the sense that the average standard deviation of the last five iterations is less than 4\% with 1000 samples. Even for the realistically sized case, the linear approximation allows us to run our algorithm for 1000 samples within 2 hours. At the same time, the linear approximation is sufficient for solution quality, i.e.\ the difference between in-sample and out-of-sample values is only 2\%. Our results, however, demonstrate that the accurate estimation of the future profit depends on not only the current reservoir level but also on the estimation of future inflows.

\bibliography{mybib}

\end{document}